\documentclass{article}

\usepackage[english]{babel}
\usepackage{amsmath,amssymb,amsthm,mathtools}
\usepackage{enumitem}
\usepackage[authoryear,round]{natbib}
\usepackage{mathtools,amscd}
\usepackage[letterpaper,top=2cm,bottom=2cm,left=3cm,right=3cm,marginparwidth=1.75cm]{geometry}

\usepackage{amsmath,amssymb, amsfonts,amsthm}
\usepackage{graphicx}
 \usepackage{dsfont}
\newtheorem{example}{Example}
\newtheorem{prop}{Proposition}
\newtheorem{lemma}[prop]{Lemma}
\newtheorem{cor}[prop]{Corollary}
\newtheorem{thm}[prop]{Theorem}
\newenvironment{remark}{{\noindent\it Remark:~}}{}

\title{ Inconsistency between Frequentist and Bayesian inference, approximation of improper priors and  permissible priors}
\author{Pierre Druilhet\footnote{\textit{Corresponding author : pierre.druilhet@uca.fr}} \ and Erwan Saint Loubert Bié,\\Laboratoire de Mathématiques Blaise Pascal, \\Université Clermont auvergne}

\begin{document}
\maketitle

\begin{abstract}
Improper priors are widely used in objective Bayesian inference, although posterior propriety does not by itself guarantee coherence between Bayesian and frequentist assessments. We investigate this issue by approximating improper priors with proper priors and examining the convergence of the resulting posterior distributions. We introduce monotone $q$-vague convergence, which ensures pointwise Kullback–Leibler (KL) convergence of the posteriors, and study expected KL (EKL) convergence together with its total-variation counterpart (ETV).
Through continuous and discrete examples, we show that EKL convergence can depend on the sequence of compact sets used for approximation. In particular, convergence along one compact exhaustion need not extend to others.

Our main result connects posterior approximation to inconsistency: strong inconsistency prevents both ETV and EKL convergence for any proper-prior approximation. Consequently, KL-permissible and TV-permissible priors cannot be strongly inconsistent. We also introduce weak inconsistency and show that it may persist despite EKL convergence.
Finally, a unified discrete framework covering the Flatland and Fraser–Monette–Ng paradoxes illustrates the distinction between these notions. In particular, ETV convergence need not imply EKL convergence, while the absence of weak inconsistency does not guarantee an EKL-convergent compact approximation.
\end{abstract}

\section{Introduction}

 Improper priors play a central role in objective Bayesian inference.
Standard constructions based on invariance, Fisher information or
reference-prior arguments often lead to measures with infinite total
mass, including flat priors for location parameters and Haar measures
on noncompact groups
\citep{jeffreys1961,bernardo1979,bergerbernardo1992,bergerbernardosun2009}.
Although these measures are not probability distributions, applying
Bayes' formula may produce proper posterior distributions. The resulting
procedures are widely used, and understanding when they can be justified
remains a fundamental question.

Posterior propriety provides only a first requirement. The
marginalization paradoxes of \citet{stonedawid1972} and
\citet{dawidstonezidek1973} showed that inference based on improper priors
can depend on how apparently equivalent calculations are performed.
Another difficulty arises when posterior and sampling assessments
exhibit a systematic incompatibility. In Stone's Flatland example, an
event has a posterior probability uniformly larger than its sampling
probability, for every observation and every parameter value
\citep{stone1976}. This
phenomenon is known as \emph{strong inconsistency} and  cannot occur under a proper prior,
because averaging the two conditional probabilities under the
corresponding joint distribution must give the same result.

A natural justification for the use of an improper prior is to interpret its formal posterior as a limit of posterior distributions obtained from proper approximating priors. A complementary probabilistic foundation for the use of improper priors is provided by the theory of conditional probability spaces, as developed by \cite{taralind2010,taralind2016}.
An alternative extension of the usual Bayesian framework is to use finitely additive probabilities, an approach advocated by \citet{defi1972} and further developed by \citet{dubins1975,heathsudderth1978,heathsudderth1989}. Finite additivity is closely related to issues of coherence, conditional probability, and non-conglomerability \citep{dubins1975,heathsudderth1978,heathsudderth1989}, while \citet{eatonsudderth1998,eatonsudderth1999,eatonsudderth2002} investigated related phenomena of consistency and strong inconsistency, particularly in invariant statistical models.
\citet{stone1982} and \citet{kada1986} proposed to relate improper priors and finitely additive priors by viewing both as limits of sequences of proper priors, under different notions of convergence. However, \citet{drusai2021} showed that these two limiting procedures capture fundamentally different features of an approximating sequence. More precisely, if a sequence of proper priors has both a finitely additive limit and an improper limit, one can construct another sequence with the same finitely additive limits but converging to any prescribed improper distribution. The reason is that an improper limit describes the behavior of the sequence within the parameter space after rescaling, whereas, independently,  finitely additive limits describe how probability mass concentrates at the boundary of the space.

These approximation arguments typically rely on pointwise convergence of the posterior distributions, that is, convergence for each fixed observation $x$. In particular, if an improper prior is restricted to an increasing sequence of compact sets, the corresponding posterior distributions converge pointwise in Kullback--Leibler divergence to the formal posterior whenever the latter is proper \citep{bergerbernardosun2009}. Such pointwise convergence is therefore too weak, by itself, to distinguish improper priors that lead to strongly inconsistent inference from those that do not.

To overcome this limitation, \citet{bergerbernardosun2009} introduced expected Kullback--Leibler (EKL) convergence, in which the KL divergence between the approximating and formal posterior distributions is averaged over the sample space with respect to the marginal distribution induced by the approximating prior. 
They  defined an improper  prior $\pi$ as \textit{permissible} if    there exists an approximating sequence of compact sets along which EKL convergence holds. Permissibility is  an explicit requirement in their formal definition of reference priors.  
  
To extend these approximation criteria beyond compact truncations, we consider convergence directly at the level of prior measures. \citet{bidr2016} introduced $q$-vague convergence by identifying Radon measures that differ only by a positive multiplicative constant, in accordance with the invariance of posterior inference under rescaling of the prior. This provides a likelihood-independent framework in which improper priors can be interpreted as limits of proper priors. However, $q$-vague convergence alone is generally too weak to ensure convergence of the corresponding posterior distributions.
However, $q$-vague convergence alone does not guarantee KL convergence of the corresponding posteriors. We therefore consider a monotone version, termed $mq$-vague convergence, which includes compact truncations and ensures KL convergence for each fixed observation whenever the limiting posterior is proper. Within this framework, we extend EKL convergence to general sequences of proper approximating priors. We also consider expected total variation (ETV) convergence, related to the convergence in probability studied by \citet{stone1970} in invariant models. This integrated criterion provides a direct way to study inconsistencies between Bayesian and frequentist inference.

Our main result establishes a direct connection between integrated posterior discrepancies and strong inconsistency. If a formal posterior is strongly inconsistent, then its ETV discrepancy from the posterior induced by any proper prior is bounded below by a strictly positive constant determined by the extent of the inconsistency. Pinsker's inequality then yields a corresponding positive lower bound for the EKL divergence. Hence, strong inconsistency precludes both ETV and EKL convergence, regardless of the approximating sequence of proper priors. In particular, a permissible prior cannot induce strongly inconsistent inference. This result provides a general justification for a relationship that had previously been suggested on the basis of examples in \citet{bergerbernardosun2009}.

The converse implications are more delicate. We introduce the notion of \emph{weak inconsistency}, which preserves a strict separation between posterior and sampling expectations. Like strong inconsistency, weak inconsistency cannot occur under a proper prior. However, we show that weak inconsistency may persist even when EKL convergence holds.
We also provide both continuous and discrete counterexamples to the conjecture, informally stated by \citet{bergerbernardosun2009}, that EKL convergence along one compact exhaustion implies convergence along every such exhaustion. Although all these truncation schemes yield pointwise convergence of the posterior distributions, the associated marginal distributions may allocate very different amounts of mass to regions where truncation has a substantial effect on the posterior. This is consistent with the fact, discussed above, that $mq$-vague convergence does not control how prior mass escapes to infinity.

 In Section \ref{section.limits}, we  develop the framework for approximating improper priors by proper priors. We compare  $q$-vague and $mq$-vague convergence at the prior level with  pointwise KL convergence of the corresponding posteriors. Then we consider the integrated criteria EKL and ETV together with the associated notions of KL- and TV-permissibility.
  In Section \ref{section.InconsitencyetELC} we review some classical examples of strong inconsistency and introduce weak inconsistency. Then, we establish our main result: strong inconsistency rules out both KL- and TV-permissibility, whereas permissibility does not preclude weak inconsistency. In section    \ref{section.generic.mmodel} we  introduce a generic discrete model encompassing the Flatland and Fraser--Monette--Ng paradoxes. This framework is used to illustrate a variety of behaviors, including permissible priors with weak inconsistency, ETV convergence without EKL convergence, dependence of EKL convergence on the choice of the approximating compact sequence, and failure of EKL-permissibility in the absence of weak inconsistency. Technical proofs and additional calculations are collected in the Appendix.

\section{Improper priors as limits of proper priors}
\label{section.limits} 

 A classical justification for the use of improper priors in Bayesian analysis is to regard them as limits of sequences of proper priors. In this view, an improper prior is acceptable when the limiting Bayes procedure obtained from proper priors coincides with the generalized Bayes procedure associated with the improper prior \citet{jeffreys1961,Berger1985,Robert2007}. This limiting argument underlies several constructions of objective priors, including Jeffreys priors and reference priors \citet{jeffreys1961,bernardo1979,kaswas1996}.
 
\medskip

\noindent Improper priors have also been discussed, from a frequentist decision-theoretic perspective, in connection with inadmissible decision rules. A standard example, due to \citet{stonedawid1972}, considers one observation
$
X\sim \mathcal N(\theta,1)
$
and the improper prior
$
\pi(\theta)\propto e^{a\theta}.
$
Under squared error loss $L(\theta,y)=(\theta-y)^2$, the corresponding generalized Bayes estimator is

\begin{equation}
	\label{eq.estimateur.nonadmissible}
	\widehat\theta(x)=x+a,
\end{equation}
which is inadmissible unless $a=0$. However, this example should not be interpreted as showing that impropriety itself causes inadmissibility. Rather, inadmissibility is a property of the resulting decision rule, relative to the sampling model and the loss function.
Indeed, the same decision rule and inadmissibility result can be obtained from a proper prior by modifying both the original loss function and prior. For instance, consider the normal proper distribution
$
\tilde\pi(\theta)\propto e^{a\theta}e^{-\theta^2/2},
$ 
 and the loss function 
$
\tilde L(\theta,y)=e^{\theta^2/2}(\theta-y)^2,
$
then the posterior risk under $(\tilde\pi,\tilde L)$ is, up to the same normalizing factor, identical to the posterior risk under $(\pi,L)$ and therefore give the same non-admissible  Bayes estimator

\medskip

In this section, let $(\pi_n)_{n\ge 1}$ be a sequence of proper priors, and let
$\pi$ denote a prior that may be either proper or improper, although in most
cases $\pi$ will be improper. By an improper distribution, we mean a Radon measure on a locally compact, Haussdorf $\sigma$-compact space, that is, a non-zero  measure that is finite on every compact set, or, in the discrete case, on every finite set. Denote by $p(x\mid \theta)$ the likelihood, and set
$$
p_n(x)=\int p(x\mid \theta)\,\pi_n(\theta)\,d\theta
$$
for the corresponding marginal likelihood. Since $\pi_n$ is proper, the posterior $\pi_n(\theta\mid x)\propto p(x\mid\theta)\,\pi_n(\theta)$ is a.s. proper. 
When $\pi$ is improper, we will assume that  $p(x)=\int p(x\mid\theta)\,\pi(\theta)\,d\theta<\infty$,  so that   the formal posterior distribution   $\pi(\theta\mid x)\propto p(x\mid\theta)\,\pi(\theta)$ is proper as a probability distribution. When $p(x)=\infty$, the posterior distribution is no longer proper. However, the posterior measure may still be defined, up to a positive multiplicative constant, by
$
\pi(\theta\mid x)\propto p(x\mid\theta),\pi(\theta),
$
provided that this expression induces a Radon measure on the parameter space \citep[see][for a more general treatment]{taraldsen2022}.

\subsection{Pointwise KL convergence of posteriors}
\label{section.logconv}

A natural way of defining  the convergence of a sequence of prior distributions to an improper one is to consider the convergence of the posterior distributions. For example, following  \citet{bergerbernardosun2009},    the sequence  $(\pi_n)$ of proper prior   approximates $\pi$ if the posterior $\pi(\cdot \mid x)$ is the KL limit of the sequence  $\pi_n(\cdot \mid x)$, that is  
\begin{equation}
\label{eq.KL.post.convergence}
	\operatorname{KL}\left(\pi_n\left(\cdot\mid x\right) \| \pi\left(\cdot\mid x\right)\right)\xrightarrow[]{n\rightarrow \infty}0
\end{equation}
where KL is the Kullback-Leibler divergence defined by

\begin{equation}
	\operatorname{KL}\left(\mu_1 \| \mu_2\right)=\int_{\Theta}{\mu_1}\left(\theta\right) \log \left(\frac{{\mu_1}\left(\theta\right)}{\mu_2\left(\theta\right)}\right) \mathrm{d} \theta.
\end{equation}
Note that this convergence is named \textit{logarithmic convergence} in  \cite{bergerbernardosun2009} and \textit{reverse KL convergence} in \cite{bergerbernardosun2024}.
As a special case, the sequence $\pi_n$ can be defined through   \textit{ approximating compact set  sequence}, that is an increasing  sequence of compacts set $\{K_n\}_n$ of $\Theta$   with $\lim_n K_n=\Theta$.  The following proposition shows that the corresponding sequence of truncated priors of a prior $\pi$  satisfies  \eqref{eq.KL.post.convergence}.
\begin{prop}[Theorem 1 in  \citet{bergerbernardosun2009}]
	\label{bbsthm1}
	Let $K_n$ be an approximating compact sequence of $\Theta$, and let $\pi$ be an improper distributions such that the posterior is proper. Define   $\pi_n(\theta)$  to be the renormalized restrictions of $\pi$ to $K_n$. Then   the sequence of posterior distributions $\pi_n(\cdot\mid x)$  KL converge   to $\pi(\cdot\mid x)$. 
\end{prop}

Although KL convergence appears to be likelihood-dependent, since it is formulated in terms of posterior distributions, Proposition~\ref{bbsthm1} shows that, for suitable sequences of proper priors $(\pi_n)_n$, the limiting behavior may in fact be likelihood-invariant. However, in some situations, the limiting prior $\pi$ need not be unique, as illustrated in the following artificial  example :

\begin{example}
	\label{example.nonuniqueness}
Let  $X=\theta+U$ where $\theta$ is an unknown non-negative integer and $
U$ random variable in $[0,1]$ uniformly distributed. For any proper prior distribution $\pi_n(\theta)>0$, the posterior distribution is  $\pi_n(\theta\mid x)=\delta_{\lfloor x \rfloor}$, the  Dirac measure on $\lfloor x \rfloor$, the integer part of $x$, which does not depends on $\pi_n$.  This also holds for any improper prior $\pi(\theta)>0$. Consequently, in that case,    any sequence $\pi_n(\theta)>0$ of proper priors KL converges   to  any improper prior  $\pi(\theta)>0$.

\end{example}
~

\subsection{ $Q$-vague and $mq$-convergence of priors}
\label{section.qvaguecv}
 
Rather than defining convergence of prior distributions through the convergence of the corresponding posteriors, \citet{bidr2016} introduced a mode of convergence, named $q$-vague convergence,  defined directly on the priors themselves, and therefore independently of both the likelihood and the resulting posterior distributions. The idea is to regard prior proper and improper distributions as elements of the projective space of Radon measures. In this section, we also introduce $mq$-vague convergence, a slightly stronger mode of convergence based on the same principle.

The starting point is the elementary observation that the  posterior inference is invariant under multiplication of the prior   by a positive constant. It is therefore natural to identify prior measures that differ only by such a multiplicative factor. This leads to the equivalence relation:
$\pi_1 \sim \pi_2$
iff
 there exists $\alpha>0 \text{ such that } \pi_2=\alpha\,\pi_1 .$
Thus, a prior can be viewed as an equivalence class, under
equivalence relation, of nonnegative Radon measures excluding the zero
measure. The corresponding quotient topology  is then called the $q$-vague topology. The related  mode of convergence can be formulated as follows: a sequence
$(\pi_n)$ is said to converge $q$-vaguely to $\pi$ if there exists a
sequence $(a_n)$ of positive real numbers such that $a_n\pi_n$ converges vaguely to $\pi$, or equivalently such that,  for every real-valued
continuous function $f$ with compact support,
\begin{equation}
\label{eq.cvqvague}
a_n\pi_n(f) \xrightarrow[n\rightarrow +\infty]{} \pi(f).
\end{equation}
An important point is that the $q$-vague limit, when it exists, is unique up to multiplication by a positive constant. Moreover, this approach is consistent with the likelihood principle, according to which proportional likelihoods should lead to the same inference.   Under suitable conditions on the likelihood, $q$-vague convergence is inherited
  by the corresponding posterior distributions, for example:

\begin{prop}[Prop. 3.1 in \citet{bidr2016}]
	\label{prop.conv.qvague.posterior}
Assume that $\theta\mapsto p(x\mid\theta)$ is a nonzero continuous function, and $\pi$ an improper distribution such that $p(x\mid\theta) \pi(\theta)$ is also a non-zero measure for $\theta$. Then   $q$-vague convergence of $\pi_n$ to $\pi$ implies $q$-vague convergence of the posterior distributions $\pi_n(\cdot\mid x)$ to $\pi(\cdot\mid x)$.
\end{prop}
\noindent  Note that Proposition~\ref{prop.conv.qvague.posterior} remains valid when the
 $\pi(\theta\mid x)$ is an improper Radon measure.  

~

Q-vague convergence may be useful to understand some Bayesian behaviors as the Jeffeys-Lindley paradox in hypothesis testing, as in the following example : 

\begin{example}
\label{exemple.jeffreyslindley}
Let $x|\theta \sim \mathcal N(\theta,1)$ and  $\pi_n =\rho\; \delta_0 + (1-\rho)\;  \mathcal N(0,n)$  for $0<\rho<1$ and $\delta_0$ is the Dirac measure in $0$. We have 
$$\pi_n\xrightarrow[n\rightarrow +\infty]{q-\textrm{vaguely}} \delta_0$$ and therefore, from Prop., \ref{prop.conv.qvague.posterior}, for any $x$, 
$$
\pi_n(\cdot\mid x) \xrightarrow[n\rightarrow +\infty]{q-\textrm{vaguely}} \delta_0.
$$
Since $p(x\mid\theta)\to 0$ as $\theta\to\pm\infty$, the posterior convergence is actually narrow \citep[][Prop.~7.1]{bidr2016}, and is therefore stronger than $q$-vague convergence.
\end{example}

~

To connect the projective-space approach with KL convergence, and subsequently with expected KL convergence, we introduce a slightly stronger notion of convergence, called monotone $q$-vague convergence, or simply $mq$-vague convergence. A sequence $(\pi_n)$ is said to converge $mq$-vaguely to $\pi$ if there exists a sequence of positive scalars $(a_n)$ such that,  for almost every $\theta$,    $a_n \pi_n(\theta)$ converge non-decreasingly to $\pi(\theta)$, i.e. 
 
\begin{equation}
	\pi_n\xrightarrow[n\rightarrow +\infty]{mq-\textrm{vaguely}} \pi  \Longleftrightarrow a_n \pi_n(\theta) \uparrow \pi(\theta)
\end{equation}
where $\pi(\theta)$ and $\pi_n(\theta)$ are densities w.r.t a common dominating $\sigma$-finite Radon  measure. In this case, $mq$-vague convergence of the priors implies $mq$-vague convergence of the corresponding posterior distributions, since the likelihood $p(x\mid\theta)$ is nonnegative. 
 The following result, which  is a direct consequence of dominated convergence theorem,  shows that $mq$-vague convergence is stronger than $q$-vague convergence, and therefore the $mq$-vague limit of $(\pi_n)$, when it exists, is also  unique, up to a scalar factor.

\begin{lemma}
	\label{lemma.mqqvague}
If a  sequence $\pi_n$  converge $mq$-vaguely to $\pi$, then it also converges  $q$-vaguely to $\pi$.
\end{lemma}

\noindent The converse part of Lemma \ref{lemma.mqqvague} is false in general as show in this example, which also shows that $q$-vague convergence does not implies KL convergence. 

\begin{example}
	\label{example.nonmqv.nonKL}
	Let $\pi_n=\rho \,\delta_0 + (1-\rho)\, \tilde \pi_n$, $0<\rho<1$, where $(\widetilde{\pi}_n)$ is a sequence of probability distributions converging $q$-vaguely to an improper distribution. It is known that $\pi_n$ converges $q$-vaguely to $\delta_0$ \citep[Prop.~2.11]{bidr2016}. The convergence is not $mq$-vague since $\pi(\theta)=\mathds{1}_{\theta=0}$, where the density is defined with respect to the dominating measure $\delta_0+\lambda$, with $\lambda$ denoting the Lebesgue measure.
	Moreover, $(\pi_n)$ is not tight and therefore cannot converge narrowly to $\delta_0$.	In particular, it cannot converge in Kullback-Leibler divergence to any proper distribution.
\end{example}

\noindent Example~\ref{example.nonmqv.nonKL} shows that $q$-vague convergence is too weak to establish a direct connection with KL convergence. Nevertheless, from the perspective of prior distributions, the Jeffreys--Lindley paradox can be naturally described in terms of $q$-vague convergence \citep[Section 7]{bidr2016}, whereas neither $mq$-vague convergence nor KL convergence applies.

~

The notion of $mq$-vague convergence is not overly restrictive, as any improper prior can be approximated in the $mq$-vague sense by a sequence of proper priors.The classical approach of approximating an improper distribution by truncation over an increasing sequence of compact sets is in fact an instance of $mq$-vague approximation, as shown by the following lemma.

\begin{lemma}
	\label{prop.mqvcKn} 
	For an improper distribution $\pi$, let $(K_n)$  be an approximating sequence of compacts. Then, the  sequence of truncated priors  $\pi_n=\frac{1}{\pi(K_n)} 1_{K_n}\pi$  converge $mq$-vaguely  to $\pi$. 
	
\end{lemma}

Truncation is not the only way to approximate $mq$-vaguely an improper prior $\pi$ and many usual approximating sequences convergence are in fact $mq$-vague approximation.  For example the sequence of normal distributions $\mathcal N(0,n)$ converges $mq$-vaguely to the flat distribution  on the real line. Similarly the sequence of beta distributions $\mathcal B(a_n,b_n)$ on $]0,1[$ converges $mq$-vaguely to the Haldane distribution, when $a_n,b_n \downarrow 0$.

  We extend Proposition~\ref{bbsthm1} by proving that $mq$-vague convergence, which is likelihood independent,   entails KL convergence of the posteriors.

 \begin{prop}
 	\label{prop.mqvaguelogcv}
 	Let  $\pi(\theta) >0$ be an improper prior such that the posterior is proper. If the sequence of proper prior $(\pi_{n})$  converges $mq$-vaguely to $\pi$, then the sequence of posteriors distributions $(\pi_{n}(\theta \mid x))$ KL converges  to $\pi(\theta \mid x)$.
 \end{prop}

 The proof can be found in  the Appendix \ref{appendix.preuvemqvaguelogcv} or in \cite{bioche2015}. Combining     Lemma \ref{prop.mqvcKn} and Prop. \ref{prop.mqvaguelogcv} gives an   alternative proof of Prop. \ref{bbsthm1}. 
  The converse of Proposition~\ref{prop.mqvaguelogcv} does not hold in general: in Example \ref{example.nonuniqueness}, let $\pi_n$ denote the Poisson distribution with parameter $n$. It is known \citep[see][]{bidr2016} that $\pi_n$ converges neither $q$-vaguely nor $mq$-vaguely to any prior $\pi(\theta)$, whereas the corresponding posteriors $\pi_n(\theta\mid x)$ KL converge  to $\pi(\theta\mid x)$.

 \medskip 
 
The main motivation for introducing the notion of $mq$-vague convergence is that it entails the $mq$-vague convergence of the   marginal distributions when the limiting marginal distribution is well-defined.

\begin{prop}
	\label{prop.marginal.convergence}
	Let  $(\pi_n)$  be a sequence that converges $mq$-vaguely  to $\pi$, and assume that the marginal $p(x)=\int p(x\mid \theta) \pi(\theta) d\theta$ is a non-zero Radon measure, then the sequence $(p_n)$ of   marginal distributions associated to $\pi_n$ converges $mq$-vaguely to $p$.
\end{prop}
 
\noindent Proposition~\ref{prop.marginal.convergence} follows directly from the monotone
convergence theorem. Note that when $\pi$ is improper, the corresponding marginal measure $p(x)$ is also improper and most often a Radon measure, though not necessarily so, see \cite{lindqvisttaraldsen2018} or example 2.19 in  \cite{bergerbernardosun2024}.



\subsection{Integrated convergence modes}
For a sequence $(\pi_n)$ of proper priors and an improper distribution $\pi$, pointwise convergence of the posteriors $\pi_n(\cdot\mid x)$ to $\pi(\cdot\mid x)$ for each fixed $x$ may be too weak when improper priors are viewed as limits of proper priors. Indeed,   it does not
control the frequentist behaviour of the resulting inference, since frequentist
assessment involves repeated sampling and hence varying values of the
observation $x$. This issue is closely related to earlier work on improper priors and limiting Bayesian procedures \citep{stone1965,stone1970,stonedawid1972,dawidstonezidek1973,dawidstone1973}. A way to treat this issue is to use an overall criterion by integrating out a distance or divergence between 
$\pi_n(\cdot\mid x)$ and $\pi(\cdot\mid x)$ w.r.t. the marginal distribution. We consider here two cases : the first one is based on the  Kullback-Leibler divergence. It was introduced by \citet{bergerbernardosun2009}  when  $\pi_n$ are defined by truncation of a improper prior by a sequence of approximating compact sequence. The second one is based on total variation distance and  was considered in the  setting of invariant models by \cite{stone1970} and \cite{wallstrom2007} under the name  \textit{probability limit} of posterior distributions. It will be useful to study inconsistency  between Bayes et frequentist inference.  
 \subsubsection{EKL convergence and KL permissible priors}
 
 Let $(\pi_n)$ be a sequence of proper priors s that converges $mq$-vaguely to  $\pi$, with   an improper distribution such that the posterior is proper. Denote  $p_n(x)=\int p(x\mid\theta)\,\pi_n(\theta)\,\mathrm{d}\theta
 $ the marginal distribution.
  The expected Kullback--Leibler (EKL) divergence between the posteriors  $\pi_n(\cdot\mid x)$ and $\pi(\cdot\mid x)$ associated to $\pi_n$ and $\pi$  is defined by 
\begin{align}
	\operatorname{EKL}(\pi_n,\pi)
	&= \int \operatorname{KL}\!\big(\pi_n(\cdot\mid x)\,\|\,\pi(\cdot\mid x)\big)\, p_n(x)\,\mathrm{d}x  
	.
\end{align}
A sequence \((\pi_n(\cdot\mid x))\) that converges $mq$-vaguely to $\pi$ is said to be   \emph{EKL convergent} \citep{bergerbernardosun2009} to \(\pi(\cdot\mid x)\) if

\begin{equation}
	\lim_{n\to\infty} \operatorname{EKL}(\pi_n,\pi)=0.
\end{equation}

When $(\pi_n)$ is obtained by restricting $\pi$ to an increasing sequence of compact sets, \citet{bergerbernardosun2009} informally conjectured that EKL convergence does not depend on the particular approximating sequence. A counterexample of this conjecture was given in \cite[Section~5]{drusai2021} for a model with a partially deterministic structure. We present here another counterexample based on the same general mechanism: the model behaves well near one boundary of the parameter space, but not near the other.  
  
  \begin{example}
  	\label{example.contreexemplenormale}
  	
  	Let $X\mid\theta\sim \mathcal N(\theta,1)$, and consider the improper prior $
  	\pi(\theta)\propto 1+e^{\theta}$, $\theta\in\mathbb R.
  	$
  Define the two sequences of approximating sequence of compact sets :  
  	$$
  	K_n^{(1)}=[-e^{2n},n],
  	\qquad
  	K_n^{(2)}=[-n,,n].
  	$$
  Let $\pi_n^{(1)}$ and $\pi_n^{(2)}$ denote the corresponding truncated versions of the prior $\pi$.  Then $\pi_n^{(1)}$ is EKL convergent to $\pi$ but not  $\pi_n^{(2)}$.   The proof can be found in Appendix  \ref{appendix.contreexemplenormale}. The failure of EKL convergence is driven by the concentration of the marginal mass $(p_n(x))$, which depends strongly on how the prior mass accumulates near the two boundaries, $-\infty$ and $+\infty$. By contrast, this boundary concentration has a much weaker effect on the KL divergence between the corresponding posterior distributions. It will be shown in Section \ref{section.weakinconsistency} that the improper prior $\pi$ used in this example leads to  weak inconsistency.

  \end{example} 
Another counterexample  in the  discrete    case  can be found in   Section~\ref{section.generic.exemple.ELC.compact}. It also follows the same scheme, but the way the mass concentrate on the boundary is more subtle.

~

An improper prior $\pi$ is said to be \emph{permissible}, in the sense of \citet{bergerbernardosun2009}, if there exists at least one approximating sequence of compact sets for which the corresponding posterior distributions is EKL convergent . In this paper, we extend this notion by allowing sequences of proper
priors that converge $mq$-vaguely to $\pi$, while retaining the
requirement of EKL convergence of the corresponding posteriors.
We refer to this extended notion as \emph{KL-permissibility}.

   Based on some examples, the authors suggested that permissibility might rule out strong inconsistency. One of the main results of this paper is to establish that this conjecture is indeed correct, as shown in Section~\ref{section.SI.vs.ELC.ensemble}. 



\subsubsection{Expected Total Variation (ETV)  convergence and TV permissible priors}

Let $\pi$ be an improper prior that yields   proper posteriors, and let $(\pi_n)$ be a sequence of proper priors that converges $mq$-vaguely to  $\pi$. We consider here an analogue of EKL convergence based on the total variation distance. This notion will be useful for establishing a connection between approximations of improper priors and strong inconsistency in Section \ref{section.InconsitencyetELC}. We define \textit{expected total variation distance} (ETV) between the posteriors $\pi_n(\cdot\mid x)$ and $\pi(\cdot\mid x)$ by
\begin{align}
	\operatorname{ETV}(\pi_n,\pi)
	&= \int \operatorname{d}_{\operatorname{TV}}\!\big(\pi_n(\cdot\mid x)\,\|\,\pi(\cdot\mid x)\big)\, p_n(x)\,\mathrm{d}x  
	.
\end{align}

\noindent where $\operatorname{d}_{\operatorname{TV}}(P,Q)$ is the total variation distance between   two probability distributions $P$ and $Q$ defined by

\begin{equation}
\label{eq.def.DTV}
\operatorname{d}_{\operatorname{TV}}(P,Q)=  \sup _{A \in \mathcal{A}}|P(A)-Q(A)|= \frac{1}{2} \sup_{ \|h\|_\infty\leq 1} \int  h(y)(p(y)-q(y)) dy.
\end{equation}

A sequence \((\pi_n(\cdot\mid x))\) that converges $mq$-vaguely to $\pi$ is said to be {ETV}  convergent  if

\begin{equation}
	\lim_{n\to\infty} \operatorname{ETV}(\pi_n,\pi)=0.
\end{equation}
Similarly to KL-permissibility, an improper prior $\pi$ is said to be
\emph{TV-permissible} if there exists a sequence of proper priors
converging $mq$-vaguely to $\pi$ such that the corresponding posteriors
converge in ETV.

The following lemma, which follows directly from Pinsker's inequality, \(\mathrm{KL}(Q \Vert P) \ge 2\,\operatorname{d}_{\mathrm{TV}}(Q,P)^2\) establishes a connection between KL-permissibility and TV-permissibility.
\begin{prop}
	\label{prop.EKLimpliesETV}
  EKL convergence implies  ETV convergence and therefore, a KL-permissible improper  prior is also TV-permissible.
\end{prop}

\subsubsection{Convergence mode when the  limit prior $\pi$ is proper}
As a complement, we study here the relationship between the convergence notions introduced above in the case when the limit distribution  $\pi$ is  proper rather than improper.

\begin{prop}
	\label{prop.ETV.EKL.proper}
	If  the sequence $\pi_n$ of  proper distributions converges $mq$-vaguely to  a proper distribution $\pi$, then the sequence of posterior $\pi_n(\theta\mid x)$ is EKL and ETV  convergent. 
\end{prop}
\begin{proof}
	Let $a_n>0$ be scalars such that $a_n\pi_n(\theta) \nearrow \pi(\theta)$. Since $a_n\pi_n(\theta) \leq \pi(\theta)$, it can be checked that  $\operatorname{KL}(\pi_n \| \pi)<+\infty$. Moreover,  we have  (see Appendix \ref{appendix.equation.KL}) 
	
		\begin{equation}\label{eq.EKLn.propre}
	\operatorname{EKL}(\pi_n,\pi) = \operatorname{KL}\left(\pi_n  \| \pi \right)
	-\operatorname{KL}\left(p_n   \|  p\right).
	\end{equation}
	and therefore  $\operatorname{EKL}(\pi_n,\pi) \leq  \operatorname{KL}\left(\pi_n  \| \pi \right)
	$. To complete the proof, we  show that  $\operatorname{KL}\left(\pi_n  \| \pi \right)$ converges to $0$: 	for any $\varepsilon$, there exists a function $0\leq f(\theta)\leq 1$ with compact support such that $\int f(\theta) d\pi(\theta)\geq 1-\varepsilon$. There also exists $N$ such that, for $n>N$,  $\int a_n f(\theta) \pi_n(\theta) \geq \int f(\theta)d\pi(\theta)-\varepsilon $. Then,
	$$
	1-2\varepsilon \leq \int a_n\; f(\theta) \; \pi_n(\theta) d\theta\leq\int a_n\;\pi_{n}(\theta) \;d\theta= a_n \leq \int \pi(\theta)\;d\theta=1
	$$	 
	and then  $	a_n$ converges to $1$.  Moreover,
	
	\begin{align*}
	\operatorname{KL(\pi_n\|\pi)} &=\int \pi_n(\theta) \log \left(\frac{1}{a_n} \frac{a_n\pi_n(\theta)}{\pi(\theta)}\right) \;d\theta\\
	&=-\log(a_n)- \int\pi_n(\theta)  \log \left(  \frac{\pi(\theta)}{a_n\pi_n(\theta)}\right) \;d\theta
	\\
	&\leq -\log(a_n)\\
	&  \xrightarrow{n\rightarrow+\infty}0
	\end{align*}

\end{proof}

The following example, inspired by the Jeffreys-Lindley paradox, considers a situation in which the priors converge $q$-vaguely but not $mq$-vaguely. The corresponding posterior distributions converge in total variation but not in Kullback-Leibler divergence. Moreover, neither EKL nor ETV convergence holds. This example therefore shows that $q$-vague convergence alone is too weak to ensure KL convergence of the posterior distributions. This is one of the reasons why $mq$-vague convergence is imposed in our definitions of EKL and ETV convergence.

\begin{example}[Example \ref{exemple.jeffreyslindley} continued]
	
	\label{example.jeffreys.nonEKL.nonETV}

Consider
$
X\mid\theta\sim\mathcal N(\theta,1)$,
$
\pi_n=\rho\delta_0+(1-\rho)\mathcal N(0,n)$,
 $0<\rho<1,
$
We have seen in Examples \ref{exemple.jeffreyslindley}  that  $\pi_n$ converges q-vaguely $\pi=\delta_0$ and the posteriors converge pointwise q-vaguely to $\delta_0$  The marginal
density is 
$
p_n(x)=\rho\varphi_1(x)+(1-\rho)\varphi_{n+1}(x),
$
and the posterior  is 

\begin{equation}
	\pi_n(\,\cdot\mid x)
=
w_n(x)\delta_0
+
\bigl(1-w_n(x)\bigr)
\mathcal N\!\left(
\frac{n}{n+1}x,\frac{n}{n+1}
\right),
\qquad \textrm{with }
w_n(x)=\frac{\rho\varphi_1(x)}{p_n(x)}.
\end{equation}

  Since $\pi_n(\,\cdot\mid x)$ is not absolutely continuous
with respect to $\pi(\,\cdot\mid x)=\delta_0$,
 $$\operatorname{KL}\bigl(
\pi_n(\,\cdot\mid x)\,\Vert\,\delta_0
\bigr)=+\infty,
$$
  and therefore there is no EKL convergence.\\
Consider now total variation distance. 
For every fixed $x$, $\varphi_{n+1}(x)\xrightarrow[n\to +\infty]{}0$, so $w_n(x)\xrightarrow[n\to +\infty]{}1$ and then 
$$
d_{\mathrm{TV}}\bigl(\pi_n(\,\cdot\mid x),\delta_0\bigr)
=1-w_n(x) \xrightarrow[n\to +\infty]{}0.
$$
Thus the posteriors converge pointwise in total variation.
However, ETV convergence   fails, since
\[
\begin{aligned}
\operatorname{ETV}(\pi_n,\delta_0)
&=\int_{\mathbb R}\bigl(1-w_n(x)\bigr)p_n(x)\,dx
=(1-\rho)\int_{\mathbb R}\varphi_{n+1}(x)\,dx 
=1-\rho.
\end{aligned}
\]

\end{example}

\section{Inconsistencies and permissible prior}
\label{section.InconsitencyetELC}

In this section, we discuss two forms of inconsistency between Bayesian and frequentist inference : strong inconsistency which is well documented in the literature and weak inconsistency, which, like strong inconsistency, cannot occur under a proper prior. Our main result is that a permissible prior cannot lead to strong inconsistency, although it may still lead to weak inconsistency. The latter case is of particular interest, since it allows us to construct examples in which EKL convergence depends on the choice of the sequence of compact sets.

\subsection{Two examples of stong inconsistency}

 We present two classical examples in discrete models. Although they arise in apparently different settings, we show that they can be represented within the same general framework, based on a common structure of the sampling model and the associated improper prior. This framework will be studied in greater details in Section~\ref{section.generic.mmodel}, where several variants will also be considered.

\subsubsection{The flatland paradox}
\label{section.flatlandmodel}
The flatland paradox was  introduced by \citet{stone1976}. It is a Bayesian version of an example of  strong inconsistency of a non-adimissible best  equivariant estimator, previously   proposed in \citet{lehm1959}.   

~

This paradox can be illustrated in various ways. Here, we give a presentation proposed by  \citet{stone1982}. Imagine a tetrahedral die being tossed an unknown number  of times $N$, which is possibly  large. The die's faces are labeled  $a$, $b$, $a^{-1}$ and $b^{-1}$. After each toss, the outcome is recorded, with the rule that if the current outcome is the inverse of the previous one, both outcomes are eliminated, that is, they cancel each other out. As a result, we obtain an non-observed  path, denoted by $\theta$, in which no consecutive inverse symbols appear. Then,  an additional toss is performed. The resulting path, denoted by $x$, is recorded according to the same rule.
For example if $\theta=(a,b,a^{-1},b)$, then  $x$ is one the the four possible paths  $(a,b,a^{-1},b,a)$, $(a,b,a^{-1},b,a^{-1})$   $(a,b,a^{-1},b,b)$ or $(a,b,a^{-1})$, each occurring with probability $1/4$,   according as the last toss is $a$, $a^{-1}$, $b$, or $b^{-1}$, respectively.  In the last case, there is \textit{annihilation}, i.e. the two final inverse symbols cancel each other, so that the length of $x$ decreases by one rather than increasing by one.

Given the observation of \( x \), a statistician aim to infer the unobserved path \( \theta \). Since this model corresponds to the free group generated by the elements \( a \) and \( b \), a flat prior on the path $\theta$ is chosen, which corresponds to the right Haar measure of the group.  

Let $A$ denote the event that annihilation occurs at the additional toss, that is, that the supplementary step is a backward step. From a frequentist perspective, conditional on $\theta$, the probability of annihilation is $1/4$, except when $\theta$ is the null path, in which case it is $0$. Hence,
$
\mathbb{P}(A\mid\theta)\leq \frac{1}{4}
\qquad \text{for all } \theta.
$

From a Bayesian perspective, under the flat prior $\pi$ on $\theta$, the posterior probability of annihilation given $x$ is $3/4$, except when $x$ is the null path, in which case it is $1$. Therefore,
$
\pi(A\mid x)\geq \frac{3}{4}
$ 
for all $x$  and  

\begin{equation}
	\sup_{\theta\in\Theta} P(A\mid \theta) \leq \inf_{x\in\mathcal{X}} \pi(A\mid x)     - \frac{1}{2},
\end{equation}

This discrepancy between the frequentist and Bayesian assessments is referred to as \textit{strong inconsistency}. Such a phenomenon cannot occur under a proper prior, since averaging $\mathbb{P}(A\mid\theta)$ with respect to $\pi(\theta)$ and $\pi(A\mid x)$ with respect to the marginal distribution $p(x)$ must give the same value. \citet{eatonsudderth1999,eatonsudderth2002} showed that strong inconsistency cannot occur when the underlying group is amenable. This condition does not hold here, since the free group is non-amenable.

One way to understand this inconsistency from an inferential perspective is to observe that the flat prior is highly informative when expressed in terms of path lengths. Let $\ell(\theta)$ denote the length of the path $\theta$. The prior induced on $\ell(\theta)$ by the flat prior on $\theta$ satisfies
$
\pi(\ell)=4\cdot  3^{\ell-1},
$
with $\pi(0)=1$ and therefore strongly favors paths of large length rather than being flat,  \citep{Hill1979,druilhet2017}. When there is annihilation, we have $\ell(x)=\ell(\theta)-1$. Therefore,  $\mathbb P( A\mid \theta ) =\mathbb P( \ell(X) = \ell(\theta)-1\mid \theta ) \leq 1/4$ for any $\theta$ whereas $\mathbb P( A\mid x ) =\pi( \ell(x) = \ell(\theta)-1 \mid x) \geq 3/4$ for any  $x$ obtained from  the highly informative prior on $\ell$.

~

\subsubsection{The Fraser, Monette, Ng paradox} 
  \label{section.frasermonettemodel}
  
  The following example of strong inconsistency was introduced by \citet{frasermonetteng1985} and later discussed by \citet{joshi1989} and \citet{GoldHowa1991} from the perspective of the likelihood principle. We show that, despite arising in a completely different setting, it follows the same underlying structure as the Flatland paradox. 
  
\noindent    For any   integers $x\geq 1$ and $\theta\geq 1$   
define the likelihood 
\begin{equation}
	 p(x \vert \theta ) = 
\left\{ 
\begin{array}{cc}
\frac{1}{3} & \text{ if }  
x \in \left\{\left[\frac{\theta}{2}\right], \, 2 \theta, \, 2 \theta +1\right\} \\[5pt]
0 & \text{ otherwise}
\end{array}
\right.
\end{equation}

where $[\theta/2]$  denote the integer part of $\theta/2$, with the convention that $[1/2]=1$.
We may note that, for $x,\theta\geq 1$, we have the equivalence~:
\begin{equation}
	\textstyle
\left\{ (x,\theta) \;;\;  x =\left[\frac{\theta}{2}\right], \, 2 \theta, \textrm{ or }  2 \theta +1  \right\}
= 
\left\{(x,\theta) \;;\;  \; \theta =\left[\frac{x}{2}\right], \, 2 x, \textrm{ or } \, 2 x +1  \right\}.
\end{equation}
Consider the flat prior $\pi(\theta)\propto 1$ on $\theta$.  Since, for any given $\theta$, the probability that $x>\theta$ is equal  to $2/3$ whereas given $x$, the same probability is equal to  $1/3$. 
Choosing  $ A =\{ (x,\theta) \;;\;   x = \left[ \theta/2\right] \}$, we have for all $\theta\geq 1$, $
P(A\mid \theta) = P (x = \left[ \theta/2\right] \mid \theta )=1/3
$ 
whereas for all $x\geq 1$,
$
\pi(A\mid x)=\pi ( \theta \in \{ 2x \, , \, 2x+1 \}  \mid x )  =2/3
$
Thus, the event $A$ has conditional probability $1/3$ for every $\theta$, while its posterior probability is $2/3$ for every $x$, providing another instance of strong inconsistency:

\begin{equation}
	\sup_{\theta\in\Theta} P(A\mid \theta) \leq \inf_{x\in\mathcal{X}} \pi(A\mid x)     - \frac{1}{3}.
\end{equation}

To see that this paradox is  closely related    to the  flatland paradox, consider  the binary   expansion of  $x$ and $\theta$. With this representation,   
the binary expansion of $x=\left[ \theta/2\right]$ is  obtained by removing the last digit in the binary expansion of $\theta$. In the same way, 
$x=2\theta$ and $x=2\,\theta +1$ are the two numbers whose binaries expansions are respectively obtained by adding a $0$ or a $1$ at the right end  of the binary expansion of $\theta$. For example, if $\theta=5$, whose binary expansion is $(101)$, then $2\theta$, $2\theta+1$, and $\left\lfloor \theta/2 \right\rfloor$ have binary expansions $(1010)$, $(1011)$, and $(10)$, respectively.

Denotes by $\ell(n ) \geq 1$ the length of the binary expansion of a non zero integer $n$.  The prior on $\ell(\theta)$ induced by the flat prior on $\theta$ is $\pi(\ell) \propto 2^{\ell-1}$. We have,    for $\theta > 1$, $\mathbb P( \ell(X) = \ell(\theta)+1 \mid \theta ) = 2/3$ while $\pi( \ell(x) = \ell(\theta)-1 \mid x) = 1/3$. . This is directly analogous to the Flatland paradox, where the flat prior on paths induced $\pi(\ell)\propto 3^{\ell-1}$ for the path lengths.

\subsection{Inconsistencies between frequentist and Bayes inference}
We define two notions of inconsistency between frequentist and Bayesian inference that generalize the two examples above. Each notion is formulated in two   ways: a functional formulation and a set-based formulation.

   We consider a statistical model $X \mid\theta \sim P_\theta$, with  $p(x\mid\theta)$ the likelihood. 
Let $\pi(\theta) $ be  a prior on $\theta$  such that the posteriors $\pi(\theta\mid x)$  are proper for all $x$. 

\subsubsection{Strong inconcistency}

\noindent
Let $ h $ be measurable   function  on $\Theta \times {\mathcal X} $ such that $\|h\|_\infty\leq 1$ with $ \|h\|_\infty =\sup_{\theta,x} |h(x)|$ and define

\begin{equation}
	\beta_h = \inf_{x\in\mathcal{X}} \int_{\Theta} h(\theta,x) \pi(\theta \mid x)\; d\theta \;,\; \alpha_h =\sup_{\theta\in\Theta} \int_{\mathcal X} h(\theta,x) p(x \mid \theta)   \;dx.
\end{equation}

The  bayesian inference induced by $\pi$ is said to be \textit{strongly inconsistent}  if there exists such a function satisfying 
$\delta_h=\beta_h - \alpha_h>0 $.
The \textit{extent of inconsistency }is defined by   
\begin{equation}
\label{eq.delta.definition}
 \delta^* = \sup_{ \|h\|_\infty\leq 1} (\beta_h - \alpha_h).
\end{equation}
 Therefore, the bayesian inference is  strongly inconsistent if and  only if  $\delta^*>0$.

 We consider  also a set-based definition of  strong inconsistency:
 let \(A\)  a measurable subset of \(\mathcal{X} \times \Theta\). Let  define the sections
 
\begin{equation}
	 A^\theta = \{x \in \mathcal{X} \; ; \; (x,\theta) \in A\}, \qquad
 A_x = \{\theta \in \Theta \; ; \; (x,\theta) \in A\},
\end{equation} 
 corresponding respectively to the slices of \(A\) at fixed \(\theta\) and fixed \(x\). Then denote 

\begin{equation}
	 \beta_A = \inf_{x\in\mathcal{X}} \pi(A_x\mid x), \qquad \alpha_A = \sup_{\theta\in\Theta} P(A^\theta \mid \theta),
\end{equation}
  The bayesian inference is said to be strongly inconsistent  if there exists a measurable  set  $A$ in  $ \Theta\times\mathcal{X}$, such that 
 $\delta_A=\beta_A - \alpha_A > 0 $.
  The \textit{extent of set-based inconsistency} is defined by 
 
\begin{equation}
	 \delta^*_{\operatorname{set}}=\sup_{A} (\beta_A- \alpha_A)
\end{equation}

\noindent To relate $\delta_{\operatorname{set}}^{*}$ and $\delta^{*}$, fix a measurable set   $A$ and  define $h(x,\theta)=\mathds 1_A(x,\theta)-1_{  A^c}(x,\theta)$. We have 
$\alpha_h=\sup_\theta\{p(A_\theta\mid \theta)-p( {A}^c_\theta\mid \theta)\}        = 2\;\alpha_A-1$. Similarly,    $\beta_h=2\;\beta_A-1$ and therefore $\beta_h- \alpha_h =2( \beta_A- \alpha_A)$. 
Since 
 $\|h\|_{\infty}\leq 1$, it follows that 
	\begin{equation}
		\delta_{\operatorname{set}}^{*}\leq \frac{1}{2}\delta^{*}.
	\end{equation}  Consequently,
 set-based strong inconsistency  implies   function-based strong inconsistency.

\subsubsection{Weak inconsistency}
\label{section.weakinconsistency}
We now introduce a weaker notion of inconsistency. As with strong inconsistency, this phenomenon cannot occur under a proper prior. The  bayesian inference induced by $\pi$ is said to be \textit{weakly inconsistent} if there exists a bounded measurable function $h$ such that
 \begin{equation}
	\label{inconsistencyweakfunction}
	\forall x,\theta \; : \; \int h(x,\theta)\;  \pi (\theta \mid x )\;d\theta > \int   h(x,\theta)  \;p(x\mid\theta) \;dx.
 \end{equation}
The corresponding set-based formulation is: there exists a measurable set $A\subseteq \Theta\times\mathcal{X}$ such that
\begin{equation}
	\label{eq.inconsistencyweakset}
	\pi(A_x\mid x)
	>
	P(A^{\theta}\mid \theta),
	\qquad
	\text{for all }( \theta,x)\in\Theta\times\mathcal{X},
\end{equation}
 Equivalently, Eq. \eqref{eq.inconsistencyweakset} can be expressed in terms of the odds ratio (OR), which is often more convenient to work with:
   \begin{equation}
 	\label{eq.inconsistencyweaksetOR}
 	\operatorname{OR}\left( \pi(A_x\mid x),P(A^\theta\mid \theta)\right)= \dfrac{\operatorname{odd}(\pi(A_x\mid x))}{\operatorname{odd(P(A^\theta\mid \theta))}}  >1, 	\qquad
 	\text{for all }( \theta,x)\in\Theta\times\mathcal{X},
 \end{equation}
with $\operatorname{odd}(p)=p/(1-p)$.

~

\noindent  Of course, strong inconsistency implies weak inconsistency. The following example shows that the converse does not hold. It also shows that a permissible prior can lead to weak inconsistency.

\begin{example}[example ~\ref{example.contreexemplenormale} continued]
\label{example.normal.weakinconsistency}
Let $X \mid \theta \sim \mathcal{N}(\theta, 1)$ and $ \pi(\theta)=1+\exp (\theta)$. From Appendix \ref{appendix.contreexemplenormale},  $p(x)=1+\exp(x+1/2)$ and  $\theta \mid x \sim w(x)\; \mathcal{N}(x, 1)+(1-w(x))\; \mathcal{N}(x+1,1)$ with $w(x)=\{1+e^{x+1 / 2}\}^{-1}$.  
Consider the set $A=\{(x,\theta)\;;\; x<\theta\}$. We have $p(A\mid \theta)= \frac{1}{2}$ whereas $\pi(A\mid x)=  \frac{1}{2}w(x)+(1-w(x)) \Phi(1)> \frac{1}{2}$ with $\Phi$ the cdf of $\mathcal N(0,1)$. Therefore, this is an example of weak inconsistency. The fact that strong inconsistency does not occur here  follows from Theorem~\ref{thm.permissible.inconsistent} since  it has been shown that the prior is permissible.
\end{example}
 
\subsection{Strong inconsistency and permissible priors }
\label{section.SI.vs.ELC.ensemble}
We have presented in Example \ref{example.contreexemplenormale},  a prior $\pi$ that leads to weak inconsistency but is  permissible. In this section, we show that a prior $\pi$ that leads to strong inconsistency is not permissible and therefore any approximating sequence of proper priors is not EKL or ETV convergent.  We first prove the result with the  Total Variation distance. Then, non EKL convergence follows as a direct consequence.

Let $\pi$ be an improper prior. We consider  a sequence $(\pi_n)$ of proper priors such that $\pi_n(\cdot\mid x)$ is KL convergent   toward  $\pi(\cdot\mid x)$. As seen above, this is the case when $(\pi_n)$ converges $mq$-vaguely to $\pi$ as for example the sequence of truncated priors $\pi_n=\mathds 1_{K_n}\cdot \pi$.

 \begin{thm}
 	\label{thm.permissible.inconsistent}
Let $\pi$ be an improper prior for which strong inconsistency occurs. Then $\pi$ is neither TV-permissible nor KL-permissible.

 \end{thm}
\begin{proof}
	The result for TV-permissibility follows directly from Proposition~\ref{prop.borneTV} below. The corresponding result for KL-permissibility then follows from Proposition~\ref{prop.EKLimpliesETV}.

\end{proof}
 
Theorem \ref{thm.permissible.inconsistent} is useful for proving the absence of strong inconsistency, as it suffices to exhibit a sequence of proper priors along which EKL or ETV convergence holds. 
\begin{prop} 
 	\label{prop.borneTV}
Let $\pi$ be an improper prior with proper prosterior. Denote    $\delta^*>0$  the extent  of inconsistency defined in \eqref{eq.delta.definition}, then  

\begin{equation}\label{eq.TV.inconsistencedelta}
\operatorname{ETV}(\pi_n,\pi) \;\ge\; \frac{1}{2}\delta^*.
\end{equation}

 \end{prop}
   \begin{proof} 
Consider  a function $h$ such that $\|h\|_{\infty}\leq 1$. By    Fubini's theorem,

\begin{align*}
   	\iint_{\Theta\times\mathcal{X}}h(\theta,x)\, \pi_n(\theta\mid x)\, p_n(x)\,\mathrm{d}\theta\;\mathrm{d} x
& 	= 
 	\int_{\Theta} 
 	\left(\int_{\mathcal{X}} 
 	h(\theta,x)\, p(x\mid\theta)\,\mathrm{d} x\right)\,\pi_n(\theta)\,\mathrm{d}\theta  
\\
&\leq \sup_{\theta\in\Theta}  \int_{\mathcal X} \;h(\theta,x)\;p(x\mid\theta)\; dx \;=\; \alpha_h.
\end{align*}

\noindent  	On the other hand,
 $$  
 	\int_{\mathcal{X}}\left(\int_{\Theta}h(\theta,x)\, \pi(\theta\mid x)\, \,\mathrm{d}\theta\right) p_n(x)\;\mathrm{d} x
 	\;\ge\; 
 	\inf_{x\in \mathcal X} \int_{\Theta} h(\theta,x) \pi(\theta \mid x)\; d\theta \;=\; \beta_h.
 $$ 	Therefore,
 	\begin{equation}
 	\label{dtv.inconsistency}
 	\int_{\mathcal{X}}\left(\int_{\Theta} h(\theta,x) \pi(\theta \mid x)\; d\theta - \int_{\Theta} h(\theta,x) \pi_n(\theta \mid x)\; d\theta\right)\; p_n(x)\; dx \;\ge\; \beta_h-\alpha_h.  
 	\end{equation}
 	By definition of the total variation distance, we have the  lower bound:
 	$$
 	\int_{\mathcal{X}} p_n(x)\,  \operatorname{d}_{\mathrm{TV}} \left( \pi_n(\cdot\mid x),\pi(\cdot\mid x)  \right) \;\mathrm{d} x \;\ge\; \frac{1}{2}(\beta_h-\alpha_h).
 	$$
 	Since this inequality holds for any $h$ with $\| h\|_\infty\leq 1$, equation \eqref{eq.TV.inconsistencedelta} follows.
 \end{proof}

 \begin{cor} 
 	\label{corol.ECL.min}	
 Let $\pi$ be an improper prior such that   $\delta^*>0$. We have the lower bound

\begin{equation}
\label{eq.EKL.maj.delta}
 \operatorname{EKL}(\pi_n,\pi)   \;\ge\; \frac{1}{2}\;\delta^{*2}.
\end{equation}

\noindent When $\pi_n$ is obtained by truncating $\pi$ to $K_n$, where $(K_n)$ is an approximating sequence of compact sets, the following sharper lower bound holds
 
 	\begin{align}
 	\label{ELC.lower.bound.2}
  \operatorname{EKL}(\pi_n,\pi)  &\;\ge\; -\ln\!\Big(1 - \frac{1}{2}\delta^*\Big) \geq \frac{1}{2}\delta^*  .  	\end{align}
 \end{cor}

 \begin{proof}
 \noindent Applying Pinsker inequality and \eqref{eq.TV.inconsistencedelta}, $\operatorname{EKL}(\pi_n,\pi)  \geq 2(\operatorname{ETV}(\pi_n,\pi))^2  \;\ge\; \frac{1}{2}\;\delta^{*2} $ which prove \eqref{eq.EKL.maj.delta}.  Consider now the case   $\pi_n\propto \mathds 1_{K_n} \cdot \pi$,  we have $\pi_n(\theta \mid x)=\frac{1}{\pi\left(K_n \mid x\right)}\pi(\theta \mid x) \mathbf{1}_{K_n}(\theta)$ and then 
  \begin{align}\label{eq.TV.compact.log}
 \operatorname{d}_{\mathrm{TV}}\big(\pi_n(\cdot\mid x),\pi(\cdot\mid x)\big) &= \pi(K_n^{\mathsf{c}}\mid x) 
 \\ 
 \mathrm{KL}\big(\pi_n(\cdot\mid x),\pi(\cdot\mid x)\big) 
 &\label{eq.KL.compact.log}= -\ln\big(1-\pi(K_n^{\mathsf{c}}\mid x)\big).
 \end{align} Applying  Jensen's inequality with the convex function  $h(u):=-\ln(1-u)$,
 \begin{align*}
 	\int_{\mathcal{X}} p_n(x)\, \mathrm{KL}\left(\pi_n(\cdot \mid x)\Vert\pi(\cdot \mid x)\right)\,\mathrm{d} x
 	&\ge\; -\ln\!\Big(1 - \int_{\mathcal{X}} p_n(x)\,\pi(K_n^{\mathsf{c}}\mid x)\,\mathrm{d} x\Big) \\
 	&=\; -\ln\!\Big(1 - \int_{\mathcal{X}} p_n(x)\,	\operatorname{d}_{\mathrm{TV}}\big(\pi_n(\cdot\mid x),\pi(\cdot\mid x)\big)\mid x)\,\mathrm{d} x\Big)   . 
 \end{align*}
 The result follows from   \eqref{eq.EKL.maj.delta}.
 \end{proof}

\begin{remark}
	Using the set-based formulation of the total variation distance given in
	\eqref{eq.def.DTV}, Formulas   \eqref{eq.TV.inconsistencedelta},
	\eqref{eq.EKL.maj.delta}, and \eqref{ELC.lower.bound.2} remain valid with $\delta^{*}$ replaced
	by $2\times\delta_{\operatorname{set}}^{*}$.
	
\end{remark}

\begin{cor}
Let $\pi_n\propto \mathds 1_{K_n} \cdot \pi$ obtained by truncating $\pi$ to $K_n$, where $(K_n)$ is an approximating sequence of compact sets. Assume that  there exists $\varepsilon>0$ and an integer $N$, such that, for all $n\geq N$ and for all $x$ 
\begin{equation}
\label{eq.cond.unif.TV}
\operatorname{d}_{\mathrm{TV}}\bigl(\pi_n(\cdot\mid x),\pi(\cdot\mid x)\bigr) \leq 1-\varepsilon.
\end{equation}	
 then   EKL convergence is equivalent to  ETV convergence.
\end{cor} 
 \begin{proof}
 From \eqref{eq.KL.compact.log} and \eqref{eq.TV.compact.log}, $\operatorname{KL}=-\log(1-d_{\operatorname{TV}})$. Since $-\log(1-x)$ is convex, when \eqref{eq.cond.unif.TV} holds, 
 $$
 \operatorname{d}_{\mathrm{TV}}\big(\pi_n(\cdot\mid x),\pi(\cdot\mid x)\big)\leq  {\mathrm{KL}}\big(\pi_n(\cdot\mid x),\pi(\cdot\mid x)\big)
 \leq \frac{-\log(\varepsilon)}{1-\varepsilon}\operatorname{d}_{\mathrm{TV}}\big(\pi_n(\cdot\mid x),\pi(\cdot\mid x)\big) 
 $$
 \end{proof}
When  condition~\eqref{eq.cond.unif.TV} fails,   the two criteria lead to different conclusions, as illustrated in Section~\ref{dvt.non.ELC}.

 \begin{remark}
Proper priors rule out such inconsistencies. However, as shown in Example~\ref{example.jeffreys.nonEKL.nonETV}, this does not imply ETV or EKL convergence for approximating sequences of priors.
 \end{remark}

\section{A generic model with a variety of cases}

 \label{section.generic.mmodel}
 
\label{etude.generic.model}

We introduce a family of discrete models whose  framework is simple, yet flexible enough to cover a broad variety of consistency and inconsistency patterns. It encompasses the Flatland paradox of Section~\ref{section.flatlandmodel} and  the Fraser et al. example of Section~\ref{section.frasermonettemodel}.    We named it the \textit{generic model}.

Let $\theta$ be a nonnegative integer. Conditional on $\theta\geq 1$, the random variable $X$ takes the values $\theta-1$ and $\theta+1$ with probabilities $1-q_\theta$ and $q_\theta$, respectively, where $0\leq q_\theta\leq 1$. In the special case $\theta=0$, $X$ takes the values $0$ and $1$ with probabilities $1-q_0$ and $q_0$, respectively. Therefore, the  likelihood  function  of the generic model  is,  for $\theta\geq 0$, 
 
 \begin{equation}
 	\label{eq.genericmodel}
 	l(\theta;x)=\left\{
 	\begin{array}{ll}
 		1-q_{\theta}   & \textrm{if } x=\theta-1 \textrm{ or } x=\theta=0 \\[5pt]
 		q_{\theta} & \textrm{if } x=\theta+1 
 	\end{array}
 	\right.
 \end{equation}
~

In this section, we consider several prior improper distributions $\pi(\theta)$. The approximating sequence of priors will be  its truncated version  $\pi_n (\theta) =  \frac{1}{\pi(K_n)}\, \pi(\theta) \;\mathbf{1}_{K_n}(\theta)$ with $K_n=\{0,1,2,...,n\}$.  

~

The Flatland paradox  corresponds to $q_\theta=3/4$ for $\theta\geq 1$, $q_0=1$, and   $\pi(\theta)\propto 3^\theta$, $\theta$ being the path length. The Fraser et al. paradox correspond to $q_\theta=2/3$, $\theta\geq 1$,  and   $\pi(\theta)\propto 2^\theta$, $\theta$ being the length of the binary expansion.

\subsection{Calculations for the Generic Model}

 \subsubsection{Posterior and marginal distributions under a prior $\pi$}

Given a prior distribution $\pi(\theta)>0$,    the improper marginal  distribution is :

\begin{equation}
\label{eq.px.generic}
p(x) = \left\lbrace 
\begin{array}{ll}
(1-q_0) \pi(0) + (1-q_1) \pi(1) & \text{ if } x=0 \\
q_{x-1}\pi(x-1) + (1-q_{x+1})\pi (x+1) & \text{ if } x>0
\end{array}
\right.
\end{equation}
and the  posterior distribution is

\begin{equation}
\label{eq.post.generic}
\pi (\theta \mid x ) = 
\left\lbrace
\begin{array}{lll}
\frac{(1-q_0)\pi(0)}{p(0)}\; \mathds 1_{\theta=0} (\theta) & + \frac{(1-q_1)\pi(1)}{p(0)} \; \mathds 1_{\theta=1} & \text{ if } x=0 \\
\frac{(1-q_{x+1})\pi(x+1)}{p(x)} \; \mathds 1_{\theta=x+1}& + 
\frac{q_{x-1}\pi(x-1)}{p(x)}\;  \mathds 1_{\theta=x-1}  &
\text{ if } x>0,
\end{array}
\right.
\end{equation}

\subsubsection{Approximation of the improper prior by truncation}

For the truncated version $\pi_n$ of $\pi$ on $K_n=\{0,1,2,...,n\}$, the marginal  distribution  is:

\begin{equation}
	p_n (x) = \left\lbrace 
\begin{array}{ll}
	(1-q_0) \pi_n (0) + (1-q_1) \pi_n (1) & \text{ if } x=0 \\
	q_{x-1}\pi_n (x-1) + (1-q_{x+1})\pi_n (x+1) & \text{ if } 0<x<n \\
	q_{n-1} \pi_n (n-1) & \text{ if } x=n \\
	q_{n} \pi_n (n) & \text{ if } x=n+1 \\
	0 &  \text{ if } x>n+1
\end{array}
\right.
\end{equation}
and the  posterior distribution is

\begin{equation}
	\pi_n (\theta \mid x ) = 
\left\lbrace
\begin{array}{ll}
	\pi(\theta \mid x) &\text{ if } 0 \leq x <n \\
	\mathds 1_{\theta=x-1}  & \text{ if } x=n+1 \text{ or } x=n.   
\end{array}
\right.
\end{equation}

\subsubsection{KL and EKL divergences between the posteriors}
We  have

\begin{equation}
	\frac{\pi_n(\theta\mid x)}{\pi(\theta \mid x)} = 
\left\lbrace
\begin{array}{ll}
	1 & \text{ if } 0 \leq x < n \\
	\frac{p(x)}{q_{x-1} \pi(x-1)} & \text{ for } x= n, n+1 \text{ and } \theta = x-1 \\
	0 & \text{ otherwise}.
\end{array}
\right.
\end{equation}

and, by  a straightforward calculation:
\begin{equation}
\operatorname{KL} \left(\; \pi_{n}(\cdot \mid x) \, \| \, \pi(\cdot \mid x)  \; \right) =
\left\{ 
\begin{array}{lcl}
0&\textrm{for} &  0 \leq x < n \textrm{ or } x>n+1\\	 
\ln \left( \frac{q_{x-1}\pi(x-1) + (1-q_{x+1})\pi (x+1)}{q_{x-1} \pi(x-1)} \right)& \textrm{for}&  x=n\textrm{ or } n+1.
\end{array}
\right.
\end{equation}

and 

\begin{equation}
	\label{EKL.genericmodel}
	\operatorname{EKL}(\pi_n,\pi)= \frac{1}{\pi(K_n)} 
	\left(
	q_{n-1} \pi(n-1) \ln \left( \frac{p(n)}{q_{n-1} \pi(n-1)} \right) +
	q_{n} \pi(n) \ln \left( \frac{p(n+1)}{q_{n} \pi(n)} \right)
	\right)
\end{equation}

\subsubsection{Odd ratios and inconsistency}

We consider the 
$A = \{ (x,\theta ) \;; \;  \theta= x+1  \}$, which provides a natural candidate for exhibiting inconsistency when it occurs.  Denote by $A^c$ the complement   of $A$. We have  
\begin{equation}
	\label{prior.modelgeneral}
	\frac{P(X=\theta-1\mid\theta)}{P(X=\theta+1\mid\theta)}=\frac{P(A^\theta \mid \theta )}{P( (A^\theta)^c \mid \theta)} = \frac{1-q_{\theta}}{q_{\theta}}, \; \text{if}\;  \theta >0  \;  \text{and, for } \theta =0 : \; \frac{P(A^\theta\mid \theta = 0)}{P( (A^\theta)^c \mid \theta = 0)} = 0.
\end{equation} 
The posterior odd for $A$ is:
\begin{equation}
	\label{posterior.modelgeneral}
	\frac{\pi ( A_x \mid x)}{\pi( A_x^c \mid x)} = \frac{1-q_{x+1}}{q_{x-1}} \times\frac{\pi(x+1)}{\pi(x-1)} \; \text {for } x \neq 0 \text{ and } \frac{\pi ( A_x \mid X=0)}{\pi(A_x^c \mid X=0)} = \frac{1-q_{1}}{1-q_{0}} \times\frac{\pi(1)}{\pi(0)} \text{ for } x=0.
\end{equation}
and the following odd ratio, which will be useful to detect inconsistency, is:

\begin{equation}
	\label{eq.OR.prior.post}
	\operatorname{OR}\left( \pi(A_x\mid x),P(A^\theta\mid \theta)\right)= \dfrac{\operatorname{odd}(\pi(A_x\mid x))}{\operatorname{odd}(P(A^\theta\mid \theta))}  = \frac{q_{\theta}}{1-q_{\theta}}
	\times 
	\frac{1-q_{x+1}}{q_{x-1}}
	\times
	\frac{\pi(x+1)}{\pi(x-1)}
\end{equation}  
and, for $\theta=0$, $	\operatorname{OR}\left( \pi(A_x\mid x),P(A^\theta\mid \theta)\right)=+\infty$.

Strong inconsistency is then equivalent to saying that, for any $(\theta,x)$, one has $\operatorname{OR}\left( \pi(A_x\mid x),P(A^\theta\mid \theta)\right) > 1+\epsilon$ for some $\varepsilon>0$.
The use of odd-ratios has the advantage to being insensitive to the multiplication of $\pi$ by a constant.

\subsection{Various cases}

\subsubsection{Strong inconsistency case} 
\label{subsection.strongincons}
 We illustrate here Theorem \ref{thm.permissible.inconsistent} by considering    the case where $q_\theta=q$ is constant with $\frac{1}{2}<q<1$ and $\pi(\theta)=\lambda^\theta$ for some $\lambda$.  This choice of prior is not intuitive  but  corresponds,  omitting the specificity  at $\theta=0$,  to the priors used in  the Flatland example  ($q=\frac{3}{4}$,  $\lambda= 3$) and  in the Fraser-Monette-Ng  example ($q=\frac{2}{3}$, $\lambda=2$).   


Put $A = \{ (x,\theta ) \;; \;  \theta= x+1  \}$, from (\ref{eq.OR.prior.post}), we have

\begin{equation}
	\operatorname{OR}\left( \pi(A_x\mid x),P(A^\theta\mid \theta)\right) =
	\left\{\begin{array}{lcl}
		\lambda^2 &\mathrm{ for } &\theta>0 \textrm{ and } x>0\\
		+\infty &\mathrm{ for } &\theta=0 \\
		\lambda\; \dfrac{q}{1-q} &\textrm{ for } &\theta>0 \textrm{ and } x=0\\
	\end{array} \right.
\end{equation}  
\begin{description}
	\item[Case 1: $\lambda <1$.] Here,  the prior is proper. Therefore, by Prop. \ref{prop.ETV.EKL.proper},  the prior is permissible and   from Theorem \ref{thm.permissible.inconsistent}, strong inconsistency does not occur.  This can also be checked directly in Formula \eqref{EKL.genericmodel}.
	\item[Case 2: $\lambda >1$.]  Here, strong inconsistency occurs for $A$. From \eqref{eq.OR.prior.post}, and since $q>\frac{1}{2}$, \\	$\operatorname{OR}\left( \pi(A_x\mid x),P(A^\theta\mid \theta)\right) \geq\lambda >1 $. Therefore, from Theorem \ref{thm.permissible.inconsistent}, $\pi$  is not permissible. For example, from (\ref{EKL.genericmodel}), we get~:	
	$$
	\operatorname{EKL}(\pi_n,\pi)= \xrightarrow{n\rightarrow +\infty}q\;\dfrac{\lambda^2-1}{\lambda^2}   \; \log \left(1+\lambda^2\dfrac{1-q}{q}\right)  \neq 0.
	$$  
	\item[Case 3 : $\lambda=1$.] The prior is  $\pi(\theta)\propto 1$. In that case, the prior is improper but also permissible and strong inconsistency does not occur. Indeed, $p(x) \propto 1$ for $x\geq 1$, $p(0)=2(1-q)$, $\pi(\theta=x-1\mid x)=q$,   and   $\pi(\theta=0 \mid x=0)=1/2$.
	Then, from (\ref{EKL.genericmodel}), we have  the EKL convergence:
	
	$$
	\operatorname{EKL}(\pi_n,\pi) =\dfrac{-2 q\log(q)}{n+1}\xrightarrow{n\rightarrow +\infty}  0
	$$

 \end{description}

\subsubsection{Permissible prior with weak inconsistency}
\label{weak.no.strong}
We illustrate here that permissibility of a prior does not rule out weak inconsistency.\\
Put $q_\theta=q>1/2$ and $\pi(\theta) = (1 + \frac{1}{\theta})^\theta$ with $\pi(0)=1$.  
From (\ref{EKL.genericmodel}),

$$
	\operatorname{EKL}(\pi_n,\pi)  \sim -2 q\log(q)/(n+1) \xrightarrow{n\rightarrow +\infty}  0
$$ 
and therefore, $\pi$ is KL permissible. Since $\mathbb P(A^\theta \mid \theta)=1$ for $\theta>0$, we have 

\begin{equation}
	\operatorname{OR}\left( \pi(A_x\mid x),P(A^\theta\mid \theta)\right)   =
	\left\{\begin{array}{lcl}
		{\left(1+\frac{1}{x+1}\right)^{x+1}} \times	 {\left(1+\frac{1}{x-1}\right)^{-(x-1)}} 	 &\mathrm{ for } &\theta>0 \textrm{ and } x>0\\
		+\infty	  &\mathrm{ for } &\theta=0 
		\\
		2\frac{q}{1-q}	  &\textrm{ for } &\theta>0 \textrm{ and } x=0\\
	\end{array} \right.
\end{equation}  
which is always greater than $1$.  Therefore, there is  weak inconsistency.

Another example is obtained by taking $\pi(\theta) = \lambda ^{\sqrt{\theta}}$ with   $\lambda > 1$. Then, it can be shown that EKL convergence holds, but, for any $(\theta,x)$, 
$$
1-q = 
P(A_{\theta} \mid \theta) < 
\pi(A_x \mid x ) = \frac{1}{1+\frac{q}{1-q} \lambda^{\sqrt{x-1} - \sqrt{x+1}}}.
$$ 
In both examples,   no strong inconsistency occurs since the prior is permissible.


\subsubsection{ETV convergence without EKL convergence}
\label{dvt.non.ELC}
By Proposition~\ref{prop.EKLimpliesETV}, EKL convergence implies ETV convergence. We   provide here  a counterexample showing that the converse does not hold. The details are given  in Appendix \ref{appendix.4.2.3}.\\
Consider $q_{\theta} = \text{min}(1,1/\theta)$ with $\displaystyle \pi(\theta) = \frac{e^{\theta^2}}{\sqrt{\theta !}}$. Then,
$$
	\operatorname{ETV}(\pi_n,\pi)=\int_{\mathcal X} p_n(x) d_{TV} ( \,\pi_n(\cdot \mid x)\; , \; \pi(\cdot \mid x)\,) \; dx \sim \frac{1}{n}.
$$
This implies that there can't be any strong inconsistency in this model. However,
$$
	\operatorname{EKL}(\pi_n,\pi)=
\int_{\mathcal X} p_n(x) \text{KL} ( \, \pi_n(\cdot \mid x)\; \Vert \; \pi(\cdot \mid x) \, )  \; dx \longrightarrow 4 ,
$$
and we don't have EKL convergence.

\subsubsection{EKL convergence  may depend on the sequence of compact set}
 \label{section.generic.exemple.ELC.compact}
We consider here an  example  where EKL convergence  depends on the choice of the compact set sequence. As in Example~\ref{example.contreexemplenormale} or in \cite {drusai2021}, Section 5, EKL convergence depends on how the mass concentrate outside compact sets

Put $q_\theta=q$  with $\frac{1}{2}<q<1$. For $\lambda>1$,  consider the  prior:\\
\phantom{Let : }
$$
\pi(\theta )=\left\{\begin{array}{cl}
	\lambda^\theta & \text { if } \theta  \text {  is even } \\
	\frac{1-q}{q} & \text { if } \theta  \text { is odd }
\end{array}\right.
$$
We consider also $\pi_n$ as the truncated of $\pi$ by the approximating compact sequences of the form
$$
K_{n_1, n_2}=\left\{0 , 2 ,4, \ldots ,2 n_1\right\} \cup\left\{1 , 3 , 5 , 7 , 9, \ldots ,2 n_2+1\right\}
$$
, which is here related to  relative rates at which $n_1$ and $n_2$ go to $+\infty$.

\begin{description}
	\item[Case 1:]  $n=n_1=n_2$. In that case, $\pi_n$ does not EKL converge  to $\pi$: 
		\begin{equation}
	\label{eq.EKL.n1=n2}
	\operatorname{EKL}(\pi_{ n},\pi) \xrightarrow[n_1\to +\infty]{} \frac{q\left(\lambda^2-1\right)}{\lambda^2} \log \left(1+\lambda^2\left(\frac{1-q}{q}\right)\right)\neq 0
	\end{equation}
We may note that (\ref{eq.EKL.n1=n2})    holds for any   $n_2$ such that   $n_2=o\left(\lambda_1^{2 n_1}\right)$.

\item[Case 2:]   For example, $n=n_1$ and $n_2= \lambda^{3n}$. In this scenario, the compact expands at a much faster rate on odd numbers compared to even numbers and we get EKL convergence:
\begin{equation}
\label{eq.EKL.n1.logn2}	\operatorname{EKL}(\pi_{ n},\pi) \sim \frac{\lambda^{2n}q^2}{\lambda^{3n}(1-q)}\log \left(1+\lambda^2\left(\frac{1-q}{q}\right)\right) \xrightarrow[n\to +\infty]{} 0  	
\end{equation}

\end{description}

 \subsubsection{Non-permissible prior with no inconsistency}
We provide an example of a prior that is not KL-permissible and yet exhibits neither  nor strong nor weak  inconsistency. Here, we consider any sequence of proper priors converging $mq$-vaguely to the improper prior, not just the compact truncations used in the previous examples.
 
 Put 
 $
 q_0=q_2=1$, $ q_k=\frac{1}{2}$  for $k\notin\{0,2\})$  and $ \pi(k)=2^k.
$
 All posteriors are proper. The idea is to  isolate the  components
 $\theta=0$ and $ X=1$:
 \[
 P(X=1\mid\theta=0)=\pi(\theta=0\mid X=1)=1,
 \]
 and therefore posterior and sampling expectations coincide at $(\theta,x)=(0,1)$
 for every bounded measurable function. Therefore Eq. \eqref{eq.inconsistencyweakset} does not hold and neither weak nor strong
 inconsistency occurs.
 
We now  prove that no $mq$-approximating sequence of compact sets yields EKL convergence.
 Let $\pi_n$ be a sequence of proper priors converging $mq$-vaguly to $\pi$, and let $p_n$ be its marginal distribution.   For $x\ge4$,
 
\begin{equation}
	 \pi(\theta=x+1\mid x)=\frac45.
\end{equation}

  Set

 \begin{equation}
 	\alpha_n=\sum_{x\ge4}p_n(x),
 \qquad
 \beta_n=\sum_{x\ge4}p_n(x)\pi_n(\theta=x+1\mid x).
 \end{equation}

    Since  $\pi_n$   converges $mq$-vaguely, and therefore $q$-vaguely, to  $\pi$ improper, then, from Prop. 2.11 in \cite{bidr2016},  $\pi_n(K)\to 0$ for any compact set $K\subset \Theta $.  Therefore,  we get   
 $
 \alpha_n \to 1$,
and $ 
 \beta_n=\frac12\pi_n(\{5,6,\ldots\})\to\frac12.
 $. 
 
 Put $r_n(x)=\pi_n(\theta=x+1 \mid x)$.
For \(x\ge4\), both posterior distributions $\pi_n(\cdot\mid x)$ and $\pi(\cdot\mid x)$ are supported on \(\{x-1,x+1\}\). Their Kullback–Leibler divergence is therefore  
 
\begin{equation}
	\operatorname{KL}\left(\pi_n(\cdot \mid x) \| \pi(\cdot \mid x)\right)=d\left(r_n(x) \| 4/5\right) . 
\end{equation}
where $d(p_1\|p_2)$ denotes the KL divergence between two   Bernoulli distributions with parameters $p_1$ and $p_2$.
 Since $\alpha_n\to1$, we have $\alpha_n>0$ for all sufficiently large $n$. For such $n$, Jensen's inequality applied to the convex function $u \mapsto d(u \| 4 / 5)$ gives
$$\begin{aligned}
	\operatorname{EKL}\left(\pi_n, \pi\right) & \geq \sum_{x \geq 4} p_n(x) d\left(r_n(x) \| 4/5\right) \\
	& \geq \alpha_n d\left(\frac{\sum_{x \geq 4} p_n(x) r_n(x)}{\alpha_n} \Big\|4/5\right) \\
	& =\alpha_n d\left(\frac{b_n}{\alpha_n} \Big\| 4/5\right) .
\end{aligned}
$$
 Consequently,
 \[
 \liminf_{n\to\infty}\operatorname{EKL}(\pi_n,\pi)
 \ge
 d\!\left(\frac12\,\middle\|\,\frac45\right)
 =\log\frac54>0.
 \]
 Hence the absence of weak inconsistency does not imply
 KL-permissibility. 
  
\section{Appendix}

\subsection{ Proof of Proposition \ref{prop.mqvaguelogcv}}
\label{appendix.preuvemqvaguelogcv}

$$
\begin{aligned}
\operatorname{KL}\left(\pi_{n}(. \mid x) \| \pi(. \mid x)\right) &=\int_{\Theta} \pi_{n}(\theta \mid x) \log \left(\frac{\pi_{n}(\theta \mid x)}{\pi(\theta \mid x)}\right) \mathrm{d} \theta \\
 &=\int_{\Theta} \pi_{n}(\theta \mid x) \log \left(\frac{a_{n} \pi_{n}(\theta)}{\pi(\theta)} \times \frac{\int_{\Theta} p(x \mid \theta) \pi(\theta) \mathrm{d} \theta}{\int_{\Theta} p(x \mid \theta) a_{n} \pi_{n}(\theta) \mathrm{d} \theta}\right) \mathrm{d} \theta \\
& \leqslant \int_{\Theta} \pi_{n}(\theta \mid x) \log \left(\frac{\int_{\Theta} p(x \mid \theta) \pi(\theta) \mathrm{d} \theta}{\int_{\Theta} p(x \mid \theta) a_{n} \pi_{n}(\theta) \mathrm{d} \theta}\right) \mathrm{d} \theta
\end{aligned}
$$
The inequality comes from  $0 \leqslant \frac{a_{n} \pi_{n}(\theta)}{\pi(\theta)} \leqslant 1$. From the monotone convergence theorem, $\lim _{n \rightarrow \infty} \int_{\Theta} p(x \mid \theta) a_{n} \pi_{n}(\theta) \mathrm{d} \theta=\int_{\Theta} p(x \mid \theta) \pi(\theta) \mathrm{d} \theta$. So, there exists $N$ such that for $n>N, \frac{\int_{\theta} p(x \mid \theta) \pi(\theta) \mathrm{d} \theta}{\int_{\Theta} p(x \mid \theta) a_{n} \pi_{n}(\theta) \mathrm{d} \theta} \leqslant 1+\varepsilon$. Then, for all $\varepsilon>0$, there exists $N$ such that for all $n>N$,
$$
\begin{aligned}
0 \leqslant \operatorname{KL}\left(\pi_{n}(. \mid x) \| \pi(. \mid x)\right) & \leqslant \int_{\Theta} \pi_{n}(\theta \mid x) \log \left(\frac{\int_{\Theta} p(x \mid \theta) \pi(\theta) \mathrm{d} \theta}{\int_{\Theta} p(x \mid \theta) a_{n} \pi_{n}(\theta) \mathrm{d} \theta}\right) \mathrm{d} \theta \\
& \leqslant \int_{\Theta} \pi_{n}(\theta \mid x) \log (1+\varepsilon) \mathrm{d} \theta \\
& \leqslant \log (1+\varepsilon)
\end{aligned}.
$$

\subsection{Details of Example  \ref{example.contreexemplenormale}}
\label{appendix.contreexemplenormale}
 Consider the model
$
X\mid \theta \sim \mathcal N(\theta,1),
$ and $
\pi(\theta)=1+e^\theta .
$ We show here that EKL depends here on the choice if the approximating sequence of compact sets and as a consequence that the prior is permissible. 
We have 
$$
p(x)=\int_{\mathbb R}\pi(\theta)\; p\;(x\mid \theta)\,d\theta
=1+\int_{\mathbb R} e^\theta  p(x\mid \theta)\,d\theta
=1+e^{x+1/2}<\infty .
$$
The formal proper posterior is 
\begin{equation}
	\label{eq.post.example4}
	 \theta \mid x \sim w(x)\; \mathcal{N}(x, 1)+(1-w(x))\; \mathcal{N}(x+1,1) 
\end{equation}
with $w(x)=\{1+e^{x+1 / 2}\}^{-1}$.
For a compact set \(K\), let $\pi_K(\theta)=\mathbf{1}_K(\theta)  \pi(\theta)/\pi(K)$. Then ,

$$
p_K(x)=\frac{1}{\pi(K)}\int_K (1+e^\theta)\varphi(x-\theta)\,d\theta .
$$
where $\varphi(y)$ is the pdf function of  $\mathcal N(0,1)$.
\subsubsection{EKL convergence depends on the compact set sequence} By a direct calculation, 
$$
\operatorname{KL}(\pi_K(\theta\mid x) \| \pi(\theta\mid x))
=
\log\left(\frac{p(x)}{p_K(x)\pi(K)}\right) .
$$
Thus  EKL discrepancy is 
$$
\operatorname{EKL}(\pi_K(\theta\mid x) \,,\,\pi(\theta\mid x))
=
\int  p_K(x) 
\log\left(\frac{p(x)}{p_K(x)\pi(K)}\right) \,dx.
$$

\subsubsection*{Case 1 : $
	K_n=[-n,n]
	$}

In that case, 
$$
\pi(K_n)
=
\int_{-n}^n(1+e^\theta)\,d\theta
=
2n+e^n-e^{-n}
\sim e^n .
$$
Hence the normalized truncated prior is asymptotically dominated by the term
\(e^\theta\), and its mass concentrates near the upper boundary \(n\). In
particular, the interval \([n-1/2,n]\) has asymptotically non-negligible prior
mass. For \(\theta\in[n-1/2,n]\) and \(X-\theta\in[0,1/2]\), one has
\(X\in[n-1/2,n+1/2]\). On such values of \(x\), 
\(p_{K_n}(x)\) misses a non-negligible part of the posterior mass lying beyond
the boundary \(n\). More precisely,
$$
\frac{p_{K_n}(x)\pi(K_n)}{p(x)}
\leq \Phi(-1/2)+o(1),
$$
where $\Phi$ is the cdf of a standard normal distribution. Therefore
$$
\log\frac{p(x)}{p_{K_n}(x)\pi(K_n)}
\geq -\log \Phi(-1/2)+o(1).
$$
Since the event considered has a probability bounded away from zero, it follows
that
$$
\liminf_{n\to\infty} \operatorname{EKL}(\pi_{C_n}(\theta\mid x) \,,\,\pi(\theta\mid x))>0.
$$
Thus EKL convergence fails for the sequence
$
C_n=[-n,n].
$.

\subsubsection*{Case 2 : $K_n=[-e^{2n},n]$}

Here
$$
\pi(K_n)
=
\int_{-e^{2n}}^n(1+e^\theta)\,d\theta
=
e^{2n}+n+e^n-e^{-e^{2n}}
\sim e^{2n}.
$$
The normalization is now dominated by the constant part \(1\) of the prior over
the very long negative interval. The mass of the exponential part of the prior is negligible,
since
$$
\frac{\int_{-e^{2n}}^n e^\theta\,d\theta}{\pi(K_n)}
\sim
\frac{e^n}{e^{2n}}
=
e^{-n}
\xrightarrow[n\to +\infty]{} 0.
$$
Consequently, the region near the upper boundary \(n\), which was responsible
for the failure of convergence for the first case, has asymptotically negligible mass.
Most of the mass lies far in the negative part of the interval, where
\(p(x)=1+e^{x+1/2}\approx 1\), and where the truncation has essentially no
effect. Hence
$$
\frac{p_{K_n}(x)\pi(K_n)}{p(x)}\xrightarrow[n\to +\infty]{}  1
$$
and therefore we will show that

\begin{equation}
\label{eq.EKL.ex.normal.K1}
\operatorname{EKL}(\pi_{K_n}(\theta\mid x) \,,\,\pi(\theta\mid x)) \xrightarrow[n\to +\infty]{}   0.
\end{equation}

\noindent Denote by $Z$ a random variable of law ${\cal N} (0,1)$. We will use the following alternative expression for EKL :
$$
\text{EKL} \; = \; 
\int_{\mathbb R} \frac{p(x)}{\pi(K_n)}  \; 
\big(  -\pi(K_n \mid x ) \ln\left(  \pi(K_n \mid x)  \right) \big) \, dx,
$$
where, here,  $p(x) = 1 + e^{x+\frac{1}{2}}$ and 
$$
\pi ( K _n \mid x) \; = \; \frac{ p_{K_n} (x)  \pi (K_n) }  {p(x)} 
\; = \; 
\frac{1}{p(x) }
\mathbb P (Z + x \in K_n)
\;  + \; 
\frac{e^{x+\frac{1}{2}} }{p(x)}
\mathbb P (Z +x+1  \in K_n ) .
$$
Denote by $\Psi(u)$ the function $\Psi \; : \; u \mapsto -u \ln(u)$, defined on $[0,1[$. We will use the inequality 
$\Psi(u) \leq \min( \sqrt{u} \, , \, \frac{1}{e} \, , \, 1-u)$ with $u = \pi(K_n \mid x)$ and according to  each possible position of $x$.

\begin{enumerate}
	
	\item[(i)] For $x\leq -e^{2n}-2$,  we have $\pi(K_n \mid x) \leq \mathbb P ( Z \geq -e^{2n} -x -1 )$ and $ p(x) \leq 2$. Thus,
	$$
	\int _{-\infty}^{-e^{2n} - 2} p(x) \Psi( \pi(K_n\mid x)  )\, dx \leq 2 \int_{1}^{+\infty} \sqrt{P(Z \geq u)} \, du 
	$$
	and, as $n \rightarrow +\infty$, 
	$$
	\int_{-\infty}^{ -e^{2n} - 2  } \frac{p(x)}{\pi(K_n)} \Psi( \pi(K_n\mid x)  )\, dx \longrightarrow 0
	$$
	
	\item[(ii)] For $-e^{2n- 2 } \leq x \leq -e^{2n} + \sqrt n +1$ we use $\Psi(\pi(K_n \mid x )) \leq \frac{1}{e}$ and $p(x) \leq 2$, yielding 
	$$
	\int _{-e^{2n} - 2}^{-e^{2n} + \sqrt n +1} p(x) \Psi( \pi(K_n\mid x)  )\, dx \leq (\sqrt n +3)\frac{2}{e}
	$$
	and, as $n \rightarrow +\infty$, 
	$$
	\int_{-e^{2n} - 2}^{-e^{2n} + \sqrt n +1} \frac{p(x)}{\pi(K_n)} \Psi( \pi(K_n\mid x)  )\, dx \longrightarrow 0.
	$$
	
	\item[(iii)] For $-e^{2n} + \sqrt n +1 \leq x \leq n-\sqrt n - 1$ we use $p(x) \leq e^{\frac{1}{2}}(1+e^x)$, with $\Psi(u) \leq 1-u$ and $\pi( K_n \mid x) \geq \mathbb P (|Z| \geq \sqrt n) \geq 1- \frac{2 e^{-n/2}}{\sqrt{2\pi n }}$. Hence
	$$
	\int _{-e^{2n} + \sqrt n +1}^{n-\sqrt n - 1} p(x) \Psi( \pi(K_n\mid x)  )\, dx \; \leq \; 
	C \frac{e^{-n/2}}{\sqrt n} \int_{K_n} p(x) dx 
	\; \leq \; 
	C \frac{e^{-n/2}}{\sqrt n} e^{\frac{1}{2}} \pi(K_n)
	$$
	for some constant $C$. Once again, as $n \rightarrow +\infty$, 
	$$
	\int_{-e^{2n} + \sqrt n +1}^{n-\sqrt n - 1} \frac{p(x)}{\pi(K_n)} \Psi( \pi(K_n\mid x)  )\, dx \longrightarrow 0.
	$$
	
	\item[(iv)] For $n -\sqrt n -1 \leq x \leq n+2$ we use $ \Psi (u ) \leq \frac{1}{e}$ and $p(x) \leq e^{n+3}$, which gives~:
	$$
	\int _{n-\sqrt n - 1} ^{n+2}  p(x) \Psi( \pi(K_n\mid x)  )\, dx 
	\; \leq \; (\sqrt n + 3)  e^{n+3}.
	$$
	Therefore, as $n \rightarrow +\infty$, 
	$$
	\int_{n-\sqrt n - 1}^{n+2}  \frac{p(x)}{\pi(K_n)} \Psi( \pi(K_n\mid x)  )\, dx \longrightarrow 0.
	$$
	
	\item[(v)] Lastly, for $x\geq n+2$, we get $\Psi(K_n \mid x) \leq \mathbb P (Z \geq x-n-1)$ and $p(x) \leq e^{n+\frac{1}{2}} e^{x-n}$. Therefore a change of variables $t=x-n$ leads to~:
	$$
	\int _{n+2}^{+\infty}  p(x) \Psi( \pi(K_n\mid x)  )\, dx 
	\; \leq 
	e^{n+\frac{1}{2}}
	\int_1^{+\infty} e^t \mathbb P (Z \geq t)\, dt. \leq C e^n
	$$
	Hence, as $n \rightarrow +\infty$, 
	$$
	\int_{n-\sqrt n - 1}^{n+2}  \frac{p(x)}{\pi(K_n)} \Psi( \pi(K_n\mid x)  )\, dx   \; \leq \; 
	\frac{C e^n}{\pi(K_n)} \; \sim \; C' e^{-n} 
	\longrightarrow 0.
	$$
\end{enumerate}

Since each of the five  terms goes to $0$, EKL convergence holds for $K_n = [-e^{2n} ; n] $.

\subsubsection{ETV convergence depends on the compact set sequence}
 
When $K_n=[-e^{2n},n]$,   \eqref{eq.EKL.ex.normal.K1} and  Prop. \ref{prop.EKLimpliesETV} implies that $\operatorname{ETV}(\pi_{K_n}(\theta\mid x) \,,\,\pi(\theta\mid x)) $ converge to $0$. So, it remains to show that there is not ETV convergence for  $	K_n=[-n,n]$:
$$
\pi_n(\theta)
=\frac{(1+e^\theta)\mathbf{1}_{[-n,n]}(\theta)}{c_n},
$$
with $c_n=2n+e^n-e^{-n}\sim e^n$.  Define 

$$
r(x)=d_{\mathrm{TV}}\!\left(
\pi_n(\cdot\mid x),\pi(\cdot\mid x)
\right)
=\pi([-n,n]^c\mid x).
$$
Using \eqref{eq.post.example4}, for every fixed $y\in\mathbb R$, with $w(x)=\{1+e^{x+1 / 2}\}^{-1}$
$$
\begin{aligned}
	r_n(n+y)
	={}&w(n+y)\bigl\{\Phi(-2n-y)+\Phi(y)\bigr\}\\
	&+(1-w(n+y))
	\bigl\{\Phi(-2n-y-1)+\Phi(y+1)\bigr\}.
\end{aligned}
$$
Because $w(n+y)\to0$, it follows that
$$
r_n(n+y)\longrightarrow\Phi(y+1).
$$
Put $X=\theta+z$ where $z\sim\mathcal N(0,1)$, we have 
$$
\operatorname{ETV}(\pi_n,\pi)
=
\frac{1}{c_n}
\int_{-n}^{n}(1+e^\theta)
\int_{\mathbb R}\varphi(z)r_n(\theta+z)\,dz\,d\theta.
$$
The contribution of the constant term in $1+e^\theta$ is bounded
above by $2n/c_n$, which tends to zero. In the exponential
contribution, the substitution $u=n-\theta$ yields
$$
\operatorname{ETV}(\pi_n,\pi)
=
\frac{e^n}{c_n}
\int_0^\infty\int_{\mathbb R}
\mathbf{1}_{\{u\le2n\}}e^{-u}\varphi(z)
r_n(n-u+z)\,dz\,du
+o(1).
$$
For each fixed $u\ge0$ and $z\in\mathbb R$,
$$
r_n(n-u+z)\longrightarrow\Phi(1-u+z).
$$
Moreover, the integrand is bounded by the integrable function
$e^{-u}\varphi(z)$. Since $e^n/c_n\to1$, dominated convergence
therefore gives
$$
\begin{aligned}
	\lim_{n\to\infty}\operatorname{ETV}(\pi_n,\pi)
	&=\int_0^\infty e^{-u}
	\int_{\mathbb R}\varphi(z)\Phi(1-u+z)\,dz\,du\\
	&=\int_0^\infty e^{-u}
	\Phi\!\left(\frac{1-u}{\sqrt2}\right)\,du.
\end{aligned}
$$
The second equality follows from the fact that the difference
of two independent standard normal variables has distribution
$\mathcal N(0,2)$.
However,  integration by parts gives 
$$
\begin{aligned}
	\int_0^\infty e^{-u}
	\Phi\!\left(\frac{1-u}{\sqrt2}\right)\,du
	&=\Phi\!\left(\frac1{\sqrt2}\right)
	-\frac1{\sqrt2}\int_0^\infty
	\varphi\!\left(\frac{1+u}{\sqrt2}\right)\,du\\
	&=2\Phi\!\left(\frac1{\sqrt2}\right)-1.
\end{aligned}
$$
since $
e^{-u}\varphi\!\left(\frac{1-u}{\sqrt2}\right)
=\varphi\!\left(\frac{1+u}{\sqrt2}\right)
$. And then  $
 	\operatorname{ETV}(\pi_n,\pi)
	\longrightarrow
	2\Phi\!\left(\frac1{\sqrt2}\right)-1
	\simeq0.5205>0.
$.
 
\subsection{Proof of Eq (\ref{eq.EKLn.propre})}
	
\label{appendix.equation.KL}

		\begin{align}
			\nonumber	\operatorname{EKL}_n& = \int_{\mathcal{X}} \operatorname{KL}\left(\pi_n (\cdot \mid x) \| \pi(\cdot \mid x)\right) p_n(x) \mathrm{d} x
			\\
			\label{eq. EKL.proba.1.1} 	&=\int_\mathcal X \left\{\int_{ \Theta} \log \left(\dfrac{\alpha \pi_n(\theta)}{\pi(\theta)}\times \dfrac{p(x)}{\alpha p_n(x)}\right) \dfrac{p(x\mid\theta)\pi_n(\theta)}{p_n(x)}   d\theta\right\} p_n(x)   dx
			\\
			&\label{eq. EKL.proba.1.2}=\int_\mathcal X \left\{\int_{ \Theta} 
			\log \left(\dfrac{\alpha \pi_n(\theta)}{\pi(\theta)}\right) p(x\mid\theta)\pi_n(\theta)   d\theta
			+
			\int_{ \Theta} \log \left( \dfrac{p(x)}{\alpha p_n(x)}\right) p(x\mid\theta)\pi_n(\theta)   d\theta\right\}     dx
			\\
			\nonumber	&=\int_\mathcal X \left\{\int_{ \Theta} 
			\log \left(\dfrac{\alpha \pi_n(\theta)}{\pi(\theta)}\right) p(x\mid\theta)\pi_n(\theta)   d\theta
			+
			\log \left( \dfrac{p(x)}{p_n(x)} \right)p_n(x) -\log(\alpha )p_n(x)\right\}     dx\quad 
			\\
			\nonumber	&\label{eq.EKLn.propre.proof}=\int_\mathcal X \left\{\int_{ \Theta} 
			\log \left(\dfrac{\alpha \pi_n(\theta)}{\pi(\theta)}\right) p(x\mid\theta)\pi_n(\theta)   d\theta\right\}  dx
			-\operatorname{KL}\left(p_n(  x) \| p(  x)\right)-\log(\alpha ) 
			\\
			&= \int_{ \Theta} \left\{\int_\mathcal X
			\log \left(\dfrac{\alpha \pi_n(\theta)}{\pi(\theta)}\right) p(x\mid\theta)\pi_n(\theta)   dx\right\} d\theta  - \log(\alpha )
			-\operatorname{KL}\left(p_n(  x) \| p(  x)\right) 
			\\
			\nonumber	&= \int_{ \Theta}  
			\log \left(\dfrac{\pi_n(\theta)}{\pi(\theta)}\right)  \pi_n(\theta)     d\theta  
			-\operatorname{KL}\left(p_n(x) \| p(x)\right)\quad  
			\\
			\nonumber	&=  \operatorname{KL}\left(\pi_n   \| \pi \right)
			-\operatorname{KL}\left(p_n(  x) \| p(  x)\right)  
		\end{align}
		
		Let's justify that we can split the integral in $\theta$  into two parts to get    (\ref{eq. EKL.proba.1.2}) from (\ref{eq. EKL.proba.1.1}). We have to prove that at least one of the two integrals is bounded.  Since   $\alpha \pi_n(\theta)\leq \pi(\theta)$ and    $\frac{-1}{e}\leq x\log(x)\leq 0$ for $x$ in $[0,1] $, we have

		\begin{align*}
			-\frac{p(x)}{\alpha \,e} 
			\leq 
			\int_{ \Theta} 
			\log \left(\dfrac{\alpha \pi_n(\theta)}{\pi(\theta)}\right) {\pi_n(\theta)}p(x\mid\theta)  d\theta
			=\frac{1}{\alpha }
			\int_{ \Theta} \log \left(\dfrac{\alpha \pi_n(\theta)}{\pi(\theta)}\right) \dfrac{\alpha \pi_n(\theta)}{\pi(\theta)}p(x\mid\theta)\pi(\theta)   d\theta  \leq   0
		\end{align*}

		Eq. (\ref{eq.EKLn.propre.proof}) follows from Fubini-Tonelli theorem since $\log \left(\dfrac{\alpha \pi_n(\theta)}{\pi(\theta)}\right)<0$. 
	
\subsection{Details of   Section \ref{section.generic.exemple.ELC.compact}}

	We give here the main steps used to obtain    \eqref{eq.EKL.n1=n2} and \eqref{eq.EKL.n1.logn2}.
	The improper marginal distribution is 
	$$
	p(x)= \left\{\begin{array}{ll}
	\frac{1-q}{q}    & \text { if }x  \text {  is even }  \phantom{\int\limits_a}\\
	\lambda^{x-1} q+\lambda^{x+1}(1-q) & \text { if } x  \text { is odd }
	\end{array}\right.
	$$and the truncated prior is  $\pi_{n_1 ,n_2}(\theta)=\frac{1}{\pi\left(K_{n_1, n_2}\right)}\; \pi(\theta) \; \mathds{1}_{K_{n_1 ,n_2}}(\theta) $ with  $\pi\left(K_{n_1 n_2}\right)=\frac{\lambda^{2(n_ 1+1)}-1}{\lambda^2-1}+(n_2+1)\left(\frac{1-q}{q}\right)$. Then, 
	
	$$p_{n_1, n_2}(x)=\left\{\begin{array}{ll}
	(1-q) \pi_{n_1, n_2}(x+1)+q \pi_{n_1, n_2}(x-1) & \text { if } x\neq 0 \phantom{\int\limits_a}\\
	\frac{1-q}{q}.  \frac{1}{\pi (K_{n_1,n_2})} & \text { if } x=0 \phantom{\int\limits_a}\\
	
\end{array}\right.
$$
and 

$$p_{n_1, n_2}(x)=\frac{1}{\pi\left(K_{n_1, n_2}\right)} \times \left\{\begin{array}{ll}
(1-q) \lambda^{x+1}+q \lambda^{x-1}    & x=2 i+1, \quad 0 \leqslant i<n_1 
\phantom{\int\limits_a}\\
\frac{1-q}{q} \cdot   &   x=2 i, \quad 0 \leqslant i \leqslant n_2 			\phantom{\int\limits_a}\\
q \lambda^{x-1}    & x=2 n_1+1
\phantom{\int\limits_a}\\
(1-q)   & x=2(n_2+1) 
\phantom{\int\limits_a}\\
0&\textrm{otherwise}
\end{array}\right.
$$
The posterior distribution for 
$$
\pi(\theta \mid x)= 
\begin{cases}
\mathds{1}_{(\theta=x+1)}(1-q)+ q \mathds{1}_  {(\theta=x-1)} & x \text{ even, non zero;}  \phantom{\int\limits_a}\\
\frac{ (1-q) \lambda^2\mathds{1}_{(\theta=x+1)}+q \mathds 1_{(\theta=x-1)}}{q+\lambda^2(1-q)} & x \text { odd }\\
\end{cases}
$$
Since $\pi_{n_1,n_2}(\theta\mid x)\neq \pi(\theta\mid x)$ only for $x=2n_1+1$ or $x=2(n_2+1)$, we get, for these values of $x$~:
$$
\pi(K_{n_1,n_2} \mid x) =
\left\lbrace
\begin{array}{ll}
\frac{  q \,  }{ (1-q)\lambda^2 +q}
& \text{ if } x = 2n_1+1  \phantom{\int\limits_a}\\
q\,   
& \text{ when } x=2n_2+2 .
\end{array}
\right.
$$
Hence, the  EKL is given according to~:   
\begin{equation}
\label{eq.EKL.n1.n2}
\operatorname{EKL}_{n_1,n_2}=\frac{\lambda^{2 n_1} q}{\pi\left(K_{n_1, n_2}\right)} \ln \left(1+\lambda^2\left(\frac{1-q}{q}\right)\right)+\frac{(1-q)}{\pi\left(K_{n_1, n_2}\right)} \ln \left(\frac{1}{q}\right)
\end{equation}
Case 1:    $n=n_1=n_2$, we have 
$
\pi( K_{n} ) \sim  \frac{ \lambda^{2n+2} }{ \lambda^2 -1} $
and we get Formula    \eqref{eq.EKL.n1=n2}.
\\
Case 2: $n=n_1$ and $n_2= \lambda^{3n}$, we get 
$
\pi( K_{n} ) \sim  \frac{1-q }{ q } \lambda^{3n}
$
which gives Formulas \eqref{eq.EKL.n1=n2} and \eqref{eq.EKL.n1.logn2}.

 \subsection{Proof of Section  \ref{dvt.non.ELC}}
 \label{appendix.4.2.3}
 	Write $c_n=\sum_{k=0}^{n}\pi(k)$ and, for $k\ge1$, set
 	\[
 	R_k
 	=
 	\frac{(1-q_{k+2})\pi(k+2)}{q_k\pi(k)}
 	=
 	\frac{k\sqrt{k+1}}{(k+2)^{3/2}}e^{4k+4}.
 	\]
 	On the marginal support, the truncated and formal posteriors
 	agree except at $x=n,n+1$. For $x=k+1$, with
 	$k\in\{n-1,n\}$, the truncated posterior is $\delta_k$,
 	whereas the formal posterior assigns probability $(1+R_k)^{-1}$
 	to $\theta=k$. Therefore,
 	\[
 	\begin{aligned}
 	\operatorname{ETV}(\pi_n,\pi)
 	&=\frac1{c_n}\sum_{k=n-1}^{n}
 	\frac{\pi(k)}{k}\frac{R_k}{1+R_k},\\
 	\operatorname{EKL}(\pi_n,\pi)
 	&=\frac1{c_n}\sum_{k=n-1}^{n}
 	\frac{\pi(k)}{k}\log(1+R_k).
 	\end{aligned}
 	\]
 	Since $\pi(k)$ is increasing,
 	\[
 	0\le\frac{c_n}{\pi(n)}-1
 	\le n\frac{\pi(n-1)}{\pi(n)}
 	=n^{3/2}e^{-2n+1}\to0.
 	\]
 	Thus $c_n\sim\pi(n)$. Moreover,
 	$R_k\sim e^{4k+4}$, so $R_k/(1+R_k)\to1$ and
 	$\log(1+R_k)\sim4k$. The contributions with $k=n-1$
 	are negligible compared with those with $k=n$, yielding
 	\[
 	\operatorname{ETV}(\pi_n,\pi)\sim\frac1n,
 	\qquad
 	\operatorname{EKL}(\pi_n,\pi)
 	\sim\frac{\log(1+R_n)}{n}\to4.
 	\]
 	Hence ETV convergence holds, excluding strong inconsistency,
 	while EKL convergence fails.

\bibliographystyle{plainnat}
\bibliography{congloparadox}
\end{document}